\documentclass[12pt]{amsart}
\usepackage[asymmetric, hmarginratio=2:3]{geometry}
\usepackage[colorlinks=true, citecolor=blue, linkcolor=blue, urlcolor=blue]{hyperref}
\usepackage{amssymb}

\def\p{\partial}
\DeclareMathOperator{\supp}{supp}

\def\C{\mathbb C}
\def\R{\mathbb R}

\def\N{\mathbb N}

\usepackage{mathtools}

\DeclarePairedDelimiter\norm{\lVert}{\rVert}

\newtheorem{theorem}{Theorem}[section]
\newtheorem{lemma}[theorem]{Lemma}
\newtheorem{proposition}[theorem]{Proposition}
\newtheorem{corollary}[theorem]{Corollary}
\theoremstyle{definition}
\newtheorem{definition}[theorem]{Definition}
\theoremstyle{remark}
\newtheorem{remark}[theorem]{Remark}

\newcommand{\abs}[1]{\lvert #1 \rvert}

\newcommand{\eps}{\epsilon}

\newcommand{\tr}{\text{Tr}}
\newcommand{\s}{\hspace{0.5pt}}
\newcommand{\ol}{\overline}
\newcommand{\op}{\overline \p}
\newcommand{\ccdot}{\,\cdot\,}
\newcommand{\z}{\overline{z}}

\title{An inverse problem for general minimal surfaces}
\author[C. Cârstea]{Cătălin I. Cârstea}
\address{National Yang Ming Chiao Tung University, Hsinchu, Taiwan}
\email{catalin.carstea@gmail.com}
\author[M. Lassas]{Matti Lassas}
\address{Department of Mathematics and Statistics, University of Helsinki, Helsinki, Finland}
\curraddr{}
\email{matti.lassas@helsinki.fi}
\author[T. Liimatainen]{Tony Liimatainen}
\address{Department of Mathematics and Statistics, University of Helsinki, Helsinki, Finland}
\curraddr{}
\email{tony.liimatainen@helsinki.fi}
\author[L. Tzou]{Leo Tzou}
\address{University of Amsterdam, Amsterdam, Netherlands}
\email{leo.tzou@gmail.com}

\begin{document}
\begin{abstract}
In this paper we consider an inverse problem of determining a minimal surface embedded in a Riemannian manifold.
We show under a topological condition that if $\Sigma$ is a $2$-dimensional embedded minimal surface, then the knowledge of the Dirichlet-to-Neumann map associated to the minimal surface equation determines $\Sigma$ up to an isometry. Without the topological condition, we show that a conformal factor of a general minimal surface $\Sigma$ can be recovered. 

We develop a semiclassical nonlinear calculus for complex geometric optics solutions, 
which allows an efficient error analysis for multiplication of the correction terms of the solutions. 
The calculus is independent of the application to the minimal surface equation and we expect it to have applications in various inverse problems for nonlinear equations in dimension $2$, in both $\R^2$ and geometric settings. Other applications of the results include generalized boundary rigidity problem and the AdS/CFT correspondence in physics.

\end{abstract}
\maketitle
\tableofcontents

\section{Minimal surface equation for general product metric}
We consider an inverse problem for $2$-dimensional minimal surfaces embedded in $3$-dimensional Riemannian manifold $(N,\overline g)$ with boundary. In this paper, we consider minimal surfaces in variational sense and define that
a $C^2$-smooth surface in  $(N,\overline g)$ is a minimal surface if its mean curvature is zero.  Let $\Sigma$ be a minimal surface. We consider minimal surfaces that are perturbations of $\Sigma$ given as graph of functions over $\Sigma$ and write the minimal surface equation in Fermi-coordinates associated to $\Sigma$. In these coordinates the metric of $(N,\overline g)$ has the form 
\begin{equation}\label{prod_metric}
 \overline g=ds^2 +g_{ab}(x,s)dx^adx^b.
\end{equation}
Here $g_s:=g(\ccdot,s)$ are a $1$-parameter family of Riemannian metrics on $\Sigma$. Here also $s$ belongs to an open interval in $\R$ and $(x^a)$ are coordinates on $\Sigma$.
For the parts of this paper that hold in general dimension we write $\dim(N)=n+1$ and $\dim(\Sigma)=n$, $n \geq 2$. We remark that Fermi-coordinates always exist and thus viewing the minimal surface equation in these coordinates is not a restriction of generality.

Let 
\[
 F(x)=(u(x),x)
\]
be the graph of a function $u:\Sigma\to \R$. 
Then the graph $F$ is a minimal surface if in Fermi-coordinates the function $u$ satisfies
\begin{equation}\label{eq:minimal_surface_general}
 -\frac{1}{\det(g_u)^{1/2}}\nabla\cdot \left( g_u^{-1}\frac{\det(g_u)^{1/2}}{\sqrt{1+\abs{\nabla u}^2_{g_u}}} \right)\nabla u + f(u,\nabla u) =0 \quad  \hbox{ on }\Sigma,
\end{equation}
where 
\[
 f(u,\nabla u)=\frac{1}{2}\frac{1}{(1+\abs{\nabla u}^2_{g_u})^{1/2}}(\p_sg_u^{-1})(\nabla u,\nabla u)+\frac{1}{2}(1+\abs{\nabla u}^2_{g_u})^{1/2}\text{Tr}(g_u^{-1}\p_sg_u).
\]
Here $\nabla$ refers to $n$ dimensional $\R^n$ Euclidean gradient and $\ccdot$ is the Euclidean inner product on $\R^n$. Throughout this paper we denote 
\[
 g_u(x)=g_{u(x)}=g(x, u(x)).
\]
We refer to the equation \eqref{eq:minimal_surface_general} as the minimal surface equation. We note that if $g(x,s)$ is independent of $s$ we have $f(u,\nabla u)=0$. In this case  the minimal surface equation is the one for the product metric $e\oplus g$ studied recently in \cite{carstea2022inverse}. Here $e$ is the Euclidean metric on $\R$. We will derive the minimal surface equation \eqref{eq:minimal_surface_general} in Section \ref{Section 2}.

We define the associated \emph{Dirichlet-to-Neumann} map (DN map in short) for \eqref{eq:minimal_surface_general} by
\begin{equation}\label{eq:DN_map_for_nonlin}
\Lambda_{\overline g} f = \left. \p_\nu u \right|_{\p \Sigma} \text{ for }  \, f \in U_\delta,
\end{equation}
where $\nu$  is the unit exterior normal vector of $\Sigma$ with respect to the metric $g$ at $(x,u(x))$,
that is, 
$$
g_{u(x)}(\nu,\nu)=1,\quad g_{u(x)}(\nu,\eta)=0,\quad \hbox{for }\nu\in T_x\partial \Sigma.
$$
In \eqref{eq:DN_map_for_nonlin}, $u$ is the unique small solution, satisfying the conditions given in Proposition \ref{prop:local_well_posedness}, to the equation \eqref{eq:minimal_surface_general}
Here $u$ is the unique small solution to \eqref{eq:minimal_surface_general} and $U_\delta=\{f\in  C^{2,\alpha}(\Sigma): \norm{f}_{C^{2,\alpha}(\p \Sigma)}\leq \delta\}$ for some $\delta>0$.  
We refer to Section \ref{Section 2} for details about the local well-posedness of \eqref{eq:minimal_surface_general} and  the DN map.

We present our main results next. To avoid various standard-like boundary determination arguments, in this paper we make the technical assumption that the quantities we consider are a priori known on the boundary. For this purpose, we need the following definition: Let $\Theta_1$ and $\Theta_2$ be tensors fields on respective Riemannian surfaces $(\Sigma_1,g_1)$ and $(\Sigma_2,g_2)$, which have a mutual boundary $\p \Sigma$. We say that $\Theta_1=\Theta_2$ to infinite order on $\p\Sigma$ if the coordinate representation of $\Theta_1$ in $g_1$-boundary normal coordinates agrees with that of $\Theta_2$ in $g_2$-boundary normal coordinates to infinite order on $\p \Sigma$. In invariant terms this means that $\mathcal{L}_{\nu_1}^l\Theta_1|_{\p \Sigma}=\mathcal{L}_{\nu_2}^l\Theta_2|_{\p \Sigma}$, $l=0,1,\ldots$, where $\mathcal{L}$ is the Lie derivative and $\nu_\beta$ is the locally defined vector field dual to the gradient of the distance to $\p \Sigma$ in $(\Sigma_\beta,g_\beta)$. 

We say that two surfaces $\Sigma_1$ and $\Sigma_2$ with a mutual boundary $\p \Sigma$ are diffeomorphic to a fixed domain in $\R^2$ by boundary fixing maps if there is a domain $\Omega\subset \R^2$ and diffeomorphisms $\Theta_\beta:\Sigma_\beta\to \Omega$, $\beta=1,2$, such that $\Theta_1|_{\p \Sigma}=\Theta_2|_{\p \Sigma}$.
\begin{theorem}\label{thm:main}
Let $(\Sigma_1, g_1)$ and $(\Sigma_2,g_2)$ be Riemannian surfaces, which are diffeomorphic to a fixed domain in $\R^2$ by boundary fixing maps. Assume that there are $1$-parameter families of Riemannian metrics $g_\beta(\ccdot,s)$, $\beta=1,2$, on $\Sigma_\beta$ and that $\Sigma_\beta$ are minimal surfaces in the sense that $0$ is a solution to \eqref{eq:minimal_surface_general} on both $\Sigma_\beta$. We also assume that $\p_s^k|_{s=0}g_1(\ccdot,s)=\p_s^k|_{s=0}g_2(\ccdot,s)$ to infinite order on $\p \Sigma$ for $k=0,1,2$.

Assume that the associated DN maps of \eqref{eq:minimal_surface_general} of manifolds $(\Sigma_1,g_1)$ and $(\Sigma_2,g_2)$ satisfy
for some $\delta>0$ sufficiently small and for all $f\in U_\delta$ 
\[
\Lambda_{g_1}f = \Lambda_{g_2}f,
\]
Then there is an isometry $F:M_1\to M_2$, 
\[
  F^*g_2=g_1,
\]
which satisfies $F|_{\p \Sigma}=\textrm{Id}$. Also $F^*\eta_2=\eta_1$, where $\eta_\beta$ are the scalar second fundamental forms of $(\Sigma_\beta,g_\beta)$. 
\end{theorem}
We note that the scalar second fundamental form $\eta_\beta(X,Y)=\langle \nabla_XN_\beta, Y\rangle_{\overline g_\beta}$ in the theorem depends also on the extrinsic geometry of $\Sigma_\beta$ in $(N_\beta, \overline g_\beta)$. Here $N_\beta$ is the unit normal to $\Sigma_\beta$, $X,Y\in T\Sigma_\beta$ and $\overline g_\beta$ as in \eqref{prod_metric}.

The reason for the assumption that $\Sigma_1$ and $\Sigma_2$ are domains in $\R^2$ up to diffeomorphisms in Theorem \ref{thm:main} is because we will refer to \cite{imanuvilov2012partial} in one part of its proof. To the best of our knowledge, the generalization of \cite{imanuvilov2012partial} to Riemannian surfaces is an open problem. Without the assumption that the minimal surface is topologically a domain in $\R^2$ we have:
\begin{theorem}\label{thm:main2}
Let $(\Sigma, g_1)$ and $(\Sigma,g_2)$ are conformally equivalent Riemannian surfaces, $g_2=cg_1$. Assume otherwise as in Theorem \ref{thm:main}, but with $\Sigma$ in place of both $\Sigma_1$ and $\Sigma_2$. 
Then 
\[
 c\equiv 1
\]
and $\eta_1=\eta_2$. 
%
%
\end{theorem}

The generalized boundary rigidity problem for minimal surfaces posed formally in \cite{Tracey} asks if areas of minimal surfaces embedded in a Riemannian manifold with boundary determine the Riemannian manifold. If the minimal surfaces are $1$-dimensional,
and thus minimizing geodesics, then the problem is the standard boundary rigidity
problem. One of the motivations for the generalized boundary rigidity problem is in
the AdS/CFT correspondence in physics, which we discuss in Section \ref{sec:adscft} below.

Let us explain how our results are related to the generalized boundary rigidity
problem. We assume that our minimal surfaces are given as solutions to the Dirichlet
problem for the minimal surface equation \eqref{eq:minimal_surface_general} in Fermi-coordinates. In this situation,
volumes of minimal surfaces determine the DN map of the minimal surface equation,
see Lemma \ref{lem:minimal surfaces and the DN map}. In general, however, the boundary of Fermi-coordinates might not
belong to the boundary of the manifold $N$, where a minimal surface is embedded. This poses a technical issue on how to pose the Dirichlet problem. For this reason, we study an exterior problem where $N$ is assumed to belong to a larger manifold $\widetilde N$. 
We refer to Section \ref{sec:areas_exterior_problem}
for details on this matter.

Let $(\Sigma,g)$ be a minimal surface embedded in $N$. \emph{The exterior problem} asks the following: If $\p \Sigma$ and the volumes of minimal surfaces $\Sigma'$ in the exterior manifold $\widetilde N$, whose boundaries satisfy $\p \Sigma'\subset \widetilde N\setminus N$, are known, do they determine the isometry type of $(\Sigma,g)$? For the exterior problem we obtain as a consequence Theorem \ref{thm:main} the following.

%
%

\begin{corollary}[Uniqueness result for the exterior problem]\label{cor:exterior_problem}
Let $(\widetilde N,\overline g)$ and $(N, \overline g|_N)$ be $3$-dimensional Riemannian manifolds with boundaries such that $N \subset\subset \widetilde N$.
Let $(\Sigma,g)$ be a $2$-dimensional minimal surface and assume it extends properly to a minimal surface $\widetilde \Sigma \subset \widetilde N$. Assume also that $ \widetilde \Sigma$ is topologically equivalent to a domain $\Omega\subset \R^2$ by a diffeomorphism $ \widetilde \phi: \widetilde \Sigma\to \Omega$. 
Let $\mathcal S$ be  is  a family of smooth  minimal surface  deformations of the surface $\widetilde \Sigma$.

Then, the knowledge of $\widetilde N\setminus N$, $\overline g|_{\widetilde N\setminus N}$, $\widetilde \Sigma\setminus \Sigma$, $\phi|_{\widetilde \Sigma\setminus \Sigma}$ and the functions $s\to (\partial \widetilde \Sigma(s),
\hbox{Vol}\s (\widetilde \Sigma(s))$, $s\in (-\delta_0,\delta_0)$, where  $\widetilde \Sigma(s)=\{\mathcal F(s,x):\s x\in \widetilde \Sigma\}$ and
$\mathcal F \in \mathcal S$,
determine 
  $(\Sigma,g)$ up to a boundary preserving isometry. The isometry also preserves the first fundamental form.
\end{corollary}
This corollary says that the knowledge of areas of minimal surfaces determine
them up to isometries. Thus, this is a result in the context of generalized
boundary rigidity problem. We mention that in this problem these $2$-dimensional minimal 
surfaces need to be glued together to construct the $3$-dimensional manifold, where the minimal surfaces are embedded.  We study how to do the gluing
in a future work.

\subsection{Sketch of the proof the main theorem}
Let us discuss how we determine a minimal surface $(\Sigma,g)$ from the DN map of \eqref{eq:minimal_surface_general} in the proof of Theorem \ref{thm:main}. A main part of the proof regards analysing products of numerous correction terms of complex geometrics optics solutions (CGOs). We call this analysis nonlinear calculus of CGOs. 

The strategy of the proof is based on the higher order linearization method that originates from \cite{kurylev2018inverse}. For this, let us consider $f_j\in C^{2,\alpha}(\p \Sigma)$, $j=1,2,3,4$ and $\alpha>0$. Let us denote by $u=u_{\eps_1f_1+ \cdots + \eps_4f_4}$ the solution of~\eqref{eq:minimal_surface_general} with boundary data $\eps_1f_1+ \cdots +\eps_4f_4$, where $\eps_j>0$ are sufficiently small parameters. We  write $\eps=0$ when referring to $\eps_1 = \cdots= \eps_4=0$. 

By taking the derivative $\p_{\eps_j}|_{\eps=0}$ of the solution $u_{\eps_1f_1+ \cdots + \eps_4f_4}$,  
we see that the function
\[
 v^j:=\frac{\p}{\p \eps_j}\Big|_{\eps=0}\s\s u_{\eps_1f_1+ \cdots + \eps_4f_4}
\]
solves the first linearized equation
\[
 \Delta_gv+qv=0,
\]
where 
\begin{align*}
 q(x)&=\frac{1}{2}\frac{d}{d s}\Big|_{s=0}\text{Tr}(g_s^{-1}\p_sg_s).
\end{align*}
One can check that the first linearized equation is the usual stability equation \cite{colding2011course} for the minimal surfaces.  
Since we know the DN map of \eqref{eq:minimal_surface_general}, we know the DN map of the first linearized equation (see Section \ref{Section 2}). The first linearized equation has a gauge symmetry in the sense that the coefficients $(g,q)$ and $(cg,c^{-1}q)$ give the same DN map. The DN map is also invariant under diffeomorphisms that are identical on the boundary. We determine $(\Sigma, g)$ up to conformal mapping by using the main result of \cite{imanuvilov2012partial}.  The rest of the proof is about determining the conformal factor and the scalar second fundamental form.

We proceed to higher order linearizations. Let us denote by
\[
 \eta(X,Y)=\langle \nabla_XN,Y\rangle_{\overline g}
\]
the scalar second fundamental form. It is a tensor operating on vector fields $X,Y$ tangential to $\Sigma$. The function $w^{jk}:=\frac{\p^2}{\p \eps_j\p \eps_k}\big|_{\eps=0}\s\s u_{\eps_1f_1+ \cdots + \eps_4f_4}$ satisfies the second linearized equation 
\[
 (\Delta_g+q)w^{jk}= \text{terms of the form } \eta(\nabla v^j, \nabla v^k)+ \text{lower order terms}. 
 \]
Lower order terms are terms that contain at most one gradient of a linearized solution $v^j$. The first aim of the second order linearization is to determine $\eta$, which is a $2\times 2$-matrix valued function. Since we know the DN map of second linearization, it follows that the integral 
\begin{multline}\label{eq:known_integral_2}
\int_{\Sigma} v^1 \eta(\nabla v^2,\nabla v^3\big)dV+\int_{\Sigma}v^2 \eta(\nabla v^1,\nabla v^3\big)dV +\int_{\Sigma} v^3 \eta(\nabla v^1,\nabla v^2\big)dV \\
+ \text{lower order and boundary terms}
 \end{multline}
is known. (We remark that in the proof we use slightly different notation where instead of $\eta$ we have $\p_sg_s^{-1}$.) We disregard lower order and boundary terms from the following argumentation.

To determine the matrix function $\eta$, we use CGOs of \cite{guillarmou2011identification} as the linearized solutions $v^j$. These are solutions to $\Delta_gv+qv=0$ given by the ansatz
\[
 e^{\Phi/h}(a+r_h),
\]
where $\Phi=\phi+i\psi$ is a holomorphic Morse function, $a$ is a holomorphic function and $r_h$ is a correction term given by 
\[
r_h=-  \overline{\p}_\psi^{-1}\sum_{j=0}^\infty T_h^j\op_\psi^{*-1}(qa).
\]
Here $\overline \p_\psi^{\s -1}$ is defined (modulo localization) by $\overline \p_\psi^{\s-1}f =\overline \p^{-1}(e^{-2i\psi/h}f)$, where $\overline \p^{-1}$ is the Cauchy-Riemann operator that solves $\overline \p^{-1}\op=\text{Id}$. We refer to details about $r_h$ to  
Section \ref{sec:CGOs} and only point out a main features of $r_h$ that its dependence in $h$ is quite complicated and it is not very small in the sense that $r_h=O_{L^2}(h^{1/2+\eps})$.  Especially $r_h$ is not given by a series of increasing powers of the small parameter $h$, which complicates the analysis in our inverse problem. 

Substitution of CGOs to the integral \eqref{eq:known_integral_2} yields a sum of the leading integral, containing only main terms $e^{\Phi/h}a$, and another integral containing combinations of the correction terms $r_h$. 
We use the leading integral to recover $h$ by stationary phase analysis after showing that the other integral containing correction terms $r_h$ is negligible when $h\to 0$. To this end, we for example need to show that
the integral
\begin{equation}\label{eq:example_error_integral}
 h^{-1}\int e^{4i\psi/h}f\op r_h=o(1),
\end{equation}
where $f$ vanishes to order $1$ at the critical point of $\psi$. Note that if $r_h$ would be independent of $h$, then the above integral is $O(1)$ by stationary phase. In a sense we are considering stationary phase for functions with a special $h$-dependence. We note that in higher dimension similar issue does not typically arise, because there correction terms of CGOs can be typically made to have correction terms of size $O_{H^k}(h^R)$ for any $k\in \N\cup \{0\}$ and $R\in \N$.

To have the required higher decay in $h$, we prove 
\begin{corollary}
 Let $f\in C_c^\infty(\Sigma)$ and $\deg(f)\geq l\geq 0$,  then 
 \[
  \int e^{4i\psi/h}f \p^l r_h=o(h^{\lfloor (\deg(f)-l)/2\rfloor +1}).
 \]
 We have analogous result for the case where $\p^l$ above is replaced by $\op^l$, or a mixture of powers of $\p$ and $\op$.
\end{corollary}
%
Here we have defined the degree of a function $f$ to be $\deg(f)=k+l$ if $f$ has an expansion $Cz^k\overline z^l+O(\abs{z}^{k+l+1})$ at the critical point, and $\lfloor\ccdot \rfloor$ is the floor function. Corollary \ref{cor:correct_decay_partial_rh} yields  \eqref{eq:example_error_integral} directly. The proof of the corollary uses the observation that $\overline \p_{\psi}^{\s -1}f$ has a certain expansion in $h$ depending on $\deg(f)$.

We call Corollary \ref{cor:correct_decay_partial_rh} and similar results in Section \ref{sec:CGOs} collectively as nonlinear CGO calculus.  We use these results numerous times, especially when considering the third linearization of the minimal surface equation. The corollary, and in fact the whole Section \ref{sec:CGOs}, are written independently of our application. We expect the nonlinear CGO calculus to have other applications in inverse problems for nonlinear equations in dimension $2$. 

After recovering the scalar second fundamental form and other unknowns appearing in the second linearization (up to the conformal mapping), we proceed to consider third linearization of the minimal surface equation. The third linearized equation is quite complicated and we only refer to Lemma \ref{Lem:Integral identity_3rd} for it here. By using stationary phase and nonlinear CGO calculus, we are ultimately able to recover the conformal factor from the third linearization. This finishes the proof of Theorem \ref{thm:main}.

\subsection{Previous literature}
The literature on inverse problems for nonlinear partial differential equations is extensive. Without doing a full review of the state of the field, we would like to mention here some early and recent works, to place our paper in historical context.
There is also a relation between our work and the AdS/CFT duality from physics. References to AdS/CFT duality are given in the next subsection below. 

A common technique for addressing inverse problems for nonlinear equations has been the so called second order linearization method. There one uses data of the form $\epsilon f$, where $\eps$ is small and $f$ a source or boundary value, and then uses the first and second order derivatives of the equation in $\epsilon$ in order to obtain separate information on the linear and nonlinear coefficients. This approach dates back to a work of Isakov and Sylvester in \cite{isakov1994global}, where the authors considered the equation
\[
-\Delta u +F(x,u)=0 
\]
on a Euclidean domain of dimension greater than or equal to three. Subject to certain constraints, they prove uniqueness for the non-linear functions $F(x,u)$. In dimension two, a similar problem was solved by Isakov and Nachman in \cite{victorN}.
Other important early works are \cite{sun2004inverse, sun2010inverse} for semilinear elliptic equations, and \cite{sun1996,sun1997inverse} for quasilinear elliptic equations. In the present work we will use some results from \cite{lassas2018poisson}, in which the inverse problem for general quasilinear equations on Riemannian manifolds is studied.

There are also many works in the nonlinear hyperbolic setting. For such equations, the nonlinearity has been shown to aid in the task if solving the inverse problems. We mention here the seminal work of Kurylev, Lassas, and Uhlmann \cite{kurylev2018inverse}, for the scalar wave equation with a quadratic nonlinearity. By using the nonlinearity, they were able to prove that local measurements determine a globally hyperbolic $4$-dimensional Lorentzian manifold up to a conformal transformation. The inverse problem for the corresponding linear wave equation is open. 

The work \cite{kurylev2018inverse} can also be said to have introduced what is now usually called the \emph{the higher order linearization method} in the study of inverse problems for nonlinear differential equations. As the name implies, this method consists in taking multiple derivatives of the nonlinear equation in question, and the associated DN map, in the small parameter $\epsilon$, in order to inductively prove uniqueness for various coefficients. The works \cite{FO19,LLLS2019inverse, feizmohammadi2023inverse} introduced the higher order linearization method for the study of inverse problems of semilinear elliptic equations on $\R^n$ and Riemannian manifolds.

Among the many works that employ the higher order linearization method we mention  \cite{MR4052205, LLLS2021b}, where the partial data problem for semilinear elliptic equations is addressed, \cite{CFKKU, kian2020partial, CaNaVa, Carstea2020, CaKa, CaFe1, CaFe2, CaGhNa, CaGhUh}, which deal with inverse problems for quasilinear elliptic equations, and \cite{lai2020partial, liimatainen2022inverse, liimatainen2022uniqueness, harrach2022simultaneous, salo2022inverse} on various other topics in the study inverse problems for nonlinear elliptic equations. We mention that in \cite{liimatainen2022uniqueness}, $0$ is not necessarily a solution, as is the case for the minimal surface equation \eqref{eq:minimal_surface_general} unless $\Sigma$ itself is a minimal surface.

\subsubsection{Inverse problems for minimal surfaces}
Let us then mention works on inverse problems for minimal surfaces. To our knowledge the first work on the subject is \cite{Tracey}. It studied inverse problem closely related to the one we study in this paper. There it was proven that if a $3$-dimensional Riemannian manifold is topologically a ball and satisfies certain curvature and foliation assumptions, the areas of a sufficiently large class of $2$-dimensional minimal surfaces determine the $3$-dimensional Riemannian manifold. Especially they determined embedded minimal surfaces that are topologically $2$-dimensional disks from the areas. As explained before Corollary \ref{cor:exterior_problem} the areas of minimal surfaces determine the DN map of the minimal surface equation \eqref{eq:minimal_surface_general}. Consequently, the relation between our work and \cite{Tracey} is that we determine more general embedded minimal surfaces, instead of those satisfying the assumptions in \cite{Tracey}, from the knowledge of the areas of the minimal surfaces.

The papers \cite{nurminen1} and \cite{nurminen2} consider an inverse problem for the minimal surface equation, for hypersurfaces in a manifold of dimension $n=3$ or higher. By denoting $e$ the metric of $\R^n$, the metric in these works is assumed to have the  form 
$c(x)e$ in \cite{nurminen1} and $c(x)(\hat g\otimes e)$ in \cite{nurminen2}. Here $\hat g$ is a simple Riemannian metric  (see \cite{nurminen2} for the definition). The quantity being determined in each of these two papers is the conformal factor $c$.

Finally, in \cite{carstea2022inverse} an inverse problem for the minimal surface equation on Riemannian manifolds of the form $\Sigma\times\R$ was studied. Here $\Sigma$ is a smooth compact two dimensional Riemannian manifold with metric $g$, and the metric $\overline g$ of $\Sigma\times\R$ is of the product form $\overline g = ds^2+g_{ab}(x)dx^adx^b$. 
It is shown in \cite{carstea2022inverse} that the DN map of the minimal surface equation determines $\Sigma$ up to an isometry. 
The current paper can be seen to be a generalization of \cite{carstea2022inverse}, see Remark \ref{rem:generalization_of_earlier}. 

\subsection{Relation to the AdS/CFT correspondence}\label{sec:adscft}
One of the motivations of this work is its quite direct relevance to the AdS/CFT correspondence, or duality, in physics. This duality is a conjectured relationship between two kinds of physical theories proposed by Maldacena \cite{maldacena1999large}. According to the duality there is a correspondence between physics of a conformal field theory (CFT) and geometric properties of an Anti-de Sitter spacetime (AdS in short). The AdS is referred to as the \emph{bulk} and it has an asymptotic infinity considered to be its boundary. The CFT is assumed to live on the boundary. A remarkable feature of this duality is that it has been successfully used to reduce various complicated quantum mechanical calculations in CFT to easier differential geometrical problems in the bulk.  The mechanism of the duality is not clear in general, see e.g. \cite{van2009comments} for a discussion and examples of the duality.

A recently proposed mechanism by Ryu and Takayanagi \cite{PhysRevLett.96.181602, ryu2006aspects} for the correspondence is the equivalence between the entanglement entropies of a CFT and areas of minimal surfaces in an AdS. Entanglement entropy is roughly speaking the experienced entropy (i.e. state of disorder) of a physical system for an observer who has only access to a subregion of a larger space. 
In their proposal, the accessible subregion $A$ is a subset of the asymptotic infinity of the AdS.  The set $A$ determines a minimal surface $\Sigma$ in the AdS by the assignment that $\Sigma$ is the minimal surface anchored on the asymptotic infinity to the boundary of $A$. By the proposal, the entanglement entropy $S_A$ of a given set $A$ in the CFT is equal to $(4G)^{-1}\text{Vol}(\Sigma)$. Here $G$ is Newton's constant and $\text{Vol}(\Sigma)$ is the volume of $\Sigma$.

Ryu and Takayanagi were able to confirm their duality, $S_A=(4G)^{-1}\text{Vol}(\Sigma)$, in several nontrivial cases. Later on, the duality has been applied to extract the geometry of bulk manifolds from the knowledge of entanglement entropies of families of sets $A$, or equivalently volumes of minimal surfaces $\Sigma$. Recall that volumes of minimal surfaces determine the DN map of the minimal surface equation \eqref{eq:minimal_surface_general} as explained before Corollary \ref{cor:exterior_problem}. Consequently, we can consider this paper to study the extraction of the bulk geometry in the AdS/CFT duality by using the DN map.

In physics literature, the metric of the bulk has been extracted from the entanglement entropies of infinite strips, circular disks and ellipsoids $A$ in \cite{bilson2008extracting, bilson2011extracting,fonda2015shape, hubeny2012extremal, jokela2021towards} and annulus shaped sets $A$ in \cite{jokela2019notes}. These works assume that the bulks have strong symmetries and are asymptotically AdS. Due to the assumptions, the reconstructions of the works reduce to simpler problems and do not require the methods of the current paper. In the mathematics literature, the extraction process was studied under various assumptions in the already discussed work \cite{Tracey}. In the recent physics paper \cite{bao2019towards}, the authors argue how the bulk can be reconstructed from volumes of minimal surfaces of codimension $\geq 2$ in general. 
With natural modifications, we expect that the current paper (dealing with codimension $1$) can be used to give rigorous justification to some of the arguments in \cite{bao2019towards}. We have separated the error term analysis from the proof of Theorem \ref{thm:main} to make the proof more accessible to a broader readership.

More generally, the current paper introduces the recent higher order linearization method (discussed above) into the study of the AdS/CFT duality between entanglement entropies and minimal surfaces. 
It seems that the method has not been used in the physics literature in this context. The dual description of the method may provide new techniques for studying the CFT side of the duality.

\subsection{Organization of the paper}

In Section \ref{Section 2} we derive equation \eqref{eq:minimal_surface_general} and state a well posedness result. In Section \ref{ss2.4} we apply the higher order linearization method to our minimal surface equation and in Section \ref{ss2.5} we derive the integral identities corresponding to the second and third linearizations. In the first part of Section \ref{sec:CGOs} we recall the construction of complex geometric optics solutions (CGOs) introduced in \cite{guillarmou2011identification}, which depend on a small asymptotic parameter $h$. We then derive what we refer to as a ``nonlinear calculus for CGOs'', which provides estimates for certain integrals containing products remainder terms, their derivatives, and  functions with various types of vanishing at the origin. In Section \ref{Section_4} we plug in the CGO solutions constructed in the previous section into the integral identities derived in Section \ref{Section 2}. We then estimate the size of all terms containing remainders, as $h\to0$. Finally, in Section \ref{sec:proof_of_main_thm} we use the leading order terms in our integral identities to prove Theorem \ref{thm:main}. In the Appendix we collect some computations and lemmas. 

\subsection{Acknowledgments}
C.C. was supported by NSTC grant number 112-2115-M-A49-002. T.L. was supported by the Academy of Finland (Centre of Excellence in Inverse Modeling and Imaging, grant numbers 284715 and 309963).
The authors wish to thank Niko Jokela and Esko Keski-Vakkuri for helpful discussions of the relation of the work to the AdS/CFT correspondence and providing references.

\section{Preliminaries}\label{Section 2}
\subsection{Local well posedness}
We record a well posedness result for the minimal surface equation.
\begin{proposition}[Local well posedness]\label{prop:local_well_posedness}
 Assume that $\text{Tr}\s (g_0^{-1}\p_sg_0)=0$ on $\Sigma$, and that $0$ is not a Dirichlet eigenvalue of the linear operator $\Delta_g+\frac{1}{2}\p_s|_{s=0} \text{Tr}\s (g_s^{-1}\p_sg_s)$ on $\Sigma$, the minimal surface equation \eqref{eq:minimal_surface_general} is well posed in the following sense. 
There exist positive constants $\delta$ and $C$ such that for any Dirichlet data $f\in U_\delta=\{f\in  C^{2,\alpha}(\partial\Sigma): \norm{f}_{C^{2,\alpha}(\p \Sigma)}\leq \delta\}$ there exists a solution $u\in C^{2,\alpha}(\Sigma)$ such that $||u||_{C^{2,\alpha}(\Sigma)}\leq C||f||_{C^{2,\alpha}(\p\Sigma)}$. The solution is unique among those  that satisfy $||u||_{C^{2,\alpha}(\Sigma)}\leq C\delta$. The correspondences $f\to u$ and $f\to\partial_\nu u|_{\p\Sigma}$ are $C^\infty$ as maps from $U_\delta$ into $C^{2,\alpha}(\Sigma)$ and $C^{1,\alpha}(\p\Sigma)$, respectively.
\end{proposition}

The proof follows along the arguments in \cite[Section 2]{LLLS2019inverse}, \cite[Appendix B]{CFKKU}. Note that the assumption $\text{Tr}(g_0^{-1}\p_sg_0)=0$ on $\Sigma$ is equivalent to the assumption that $u\equiv 0$ is a solution to the minimal surface equation \eqref{eq:minimal_surface_general}. Either of these assumptions are in turn equivalent to $\Sigma$ having zero mean curvature (see e.g. \cite[Section 2]{lassas2016calder} for formulas about mean curvature in Fermi-coordinates). The operator $\Delta_g+\frac{1}{2}\p_s|_{s=0} \text{Tr}(g_s^{-1}\p_sg_s)$ is the linearized operator of the minimal surface equation. In the context of minimal surfaces it is known as the stability operator (see \cite[Section 8]{colding2011course}) and the condition that $0$ is not its Dirichlet eigenvalue related to the stability of the minimal surface $\Sigma$. To see that $\Delta_g+\frac{1}{2}\p_s|_{s=0} \text{Tr}(g_s^{-1}\p_sg_s)$ is the stability operator, one can check that $\frac{1}{2}\p_s|_{s=0} \text{Tr}(g_s^{-1}\p_sg_s)=-\abs{A}^2-\text{\text{Ric}}(N,N)$, where $\abs{A}$ is the norm of the second fundamental form, $\text{Ric}$ is the Ricci curvature and $N$ normal vector field to $\Sigma$.

\subsection{Derivation of the minimal surface equation}
Let us derive the minimal surface equation \eqref{eq:minimal_surface_general}. This is the equation a minimal surface satisfies in Fermi-coordinates. For possible future references, we consider in this section $n$-dimensional minimal surfaces embedded in $n+1$-dimensional Riemannian manifolds, $n\geq 2$.

We assume that the minimal surfaces are given as graphs over a submanifold $\Sigma$. Let us choose Fermi coordinates for a neighborhood $N=I\times \Sigma$ of $\Sigma$. Here $I$ is an interval in $\R$ containing $0$.  In Fermi coordinates the metric is of the form
\begin{equation*}
 \overline g=ds^2 +g_{ab}(x,s)dx^adx^b,
\end{equation*}
where $g(x,s)$ is a Riemannian metric on $\Sigma$ for all $s$.  $N=I\times \Sigma$ for all $s$. Here and below we also use Einstein summation over repeated indices.   We consider $g(\ccdot,s)=g_s$ as $1$-parameter family of Riemannian metrics on $\Sigma$. We have $\dim(N)=n+1$ and $\dim(\Sigma)=n$.

Let 
\[
 F(x)=(u(x),x)
\]
be the graph of a function $u:\Sigma\to I\subset \R$. The volume form of the graph $Y$ of $F$ is given by the determinant of the induced metric on the graph $Y$. Coordinates for $Y$ are given by coordinates on $\Sigma$ by
\[
 x\mapsto (u(x),x).
\]
Let $(x^k)_{k=1}^n$ be the above coordinates on $Y$ and let $\p_k$ be the corresponding coordinate vectors. Y .
For simplicity, let us assume that $(x^k)_{k=1}^n$ are global coordinates on $\Sigma$. The general case when
$\Sigma$ is covered with finitely many coordinate charts can be considered using a suitable partition
of unity. Let us also denote by $h_{jk}$ the induced metric on $Y$:    
\[
 h_{jk}(x)=\overline g_{F(x)}(F_*\p_{x_j},F_*\p_{x_k}),
\]
where $j,k=1,\ldots, n$. Here we have that 
\[
 F_*\p_{x_j}=DF_j^a\p_a,
\]
where $a=0,1,\ldots, n$. Note that if $a\neq 0$, then $DF_j^a=\delta_j^a$. We also have $DF_j^0=\p_ju$. It follows that the induced metric on $Y$ reads
\begin{align*}
 h_{jk}&=DF_j^aDF_k^b g_{F(x)}(\p_a,\p_b)=DF_j^0DF_k^0 \s\overline g_{00}|_{F(x)}+\sum_{\alpha,\beta=1}^nDF_j^\alpha DF_k^\beta \overline g_{F(x)}(\p_\alpha,\p_\beta) \\
 &=\p_ju(x)\p_ku(x)+g_{jk}|_{F(x)}=\p_ju(x)\p_ku(x)+g(x,u(x))_{jk}.
\end{align*}

The volume of $Y$ is 
\[
 \text{Vol}(Y)=\int_\Sigma \sqrt{\det(h)}dx^1\wedge\cdots \wedge dx^n.
\]
Using the formula for $h_{jk}$ we have that
\[
 \det(h)=\det\Big(\nabla u\otimes \nabla u +g(x,u(x))\Big)=\det(g(x,u(x)))\det\Big(I+(g(x,u)^{-1}\nabla u)\otimes \nabla u\Big).
\]
By \cite[Lemma 1.1]{Ding}, we have 
\[
 \det\Big(I+(g(x,u)^{-1}\nabla u)\otimes \nabla u\Big)=1+(g(x,u)^{-1}\nabla u) \cdot \nabla u=1+\abs{\nabla u}_{g(x,u(x))}^2.
\]
Finally, the volume of $Y$ equals
\[
 \text{Vol}(Y)=\int_\Sigma \sqrt{1+\abs{\nabla u}^2_{g(x,u)}}\det(g(x,u))^{1/2}dx^1\wedge\cdots \wedge dx^n.
\]

Let us then compute the minimal surface equation. For that, we consider a variation
\[
 Y(u+tv):=\{(u(x)+tv(x),x): x\in \Sigma\} \subset N
\]
of the surface $Y$, where $v:\Sigma \to \R$ is a smooth function. We denote the volume of $Y(u+tv)$ by $\text{Vol}(u+tv)$. Then, 
\begin{align*}
 &\frac{d}{dt}\Big|_{t=0}\text{Vol}(u+tv)=\frac 12 \int_\Sigma \frac{\det(g(x,u))^{1/2}}{\sqrt{1+\abs{\nabla u}^2_{g(x,u)}}}\frac{d}{dt}\Big|_{t=0}\left(\abs{\nabla (u+tv)}^2_{g(x,u+tv)}\right) \\
 &+\int_\Sigma \sqrt{1+\abs{\nabla u}^2_{g(x,u)}} \frac{d}{dt}\Big|_{t=0}\det(g(x,u+tv))^{1/2}.
\end{align*}
We have
\begin{align*}
 &\frac{d}{dt}\Big|_{t=0}\abs{\nabla (u+tv)}^2_{g(x,u+tv)}=2g(x,u)^{-1}(\nabla u, \nabla v)+v\p_s(g^{-1})(x,u)(\nabla u,\nabla u) \\
\end{align*}
and
\begin{align*}
 &\frac{d}{dt}\Big|_{t=0}\det(g(x,u+tv))^{1/2}=\frac{1}{2}\det(g(x,u))^{1/2}\text{Tr}(g(x,u)^{-1}v\p_sg(x,u)).
\end{align*}
Thus
\begin{multline}\label{eq:first_var}
 \frac{d}{dt}\Big|_{t=0}\text{Vol}(u+tv)=\frac 12 \int_\Sigma \frac{\det(g(x,u))^{1/2}}{\sqrt{1+\abs{\nabla u}^2_{g(x,u)}}}\Big(2g(x,u)^{-1}(\nabla u, \nabla v)+\p_s(g^{-1})(x,u)(\nabla u,\nabla u)v\Big) \\
 +\int_\Sigma \sqrt{1+\abs{\nabla u}^2_{g(x,u)}}\frac{1}{2}\det(g(x,u))^{1/2}\text{Tr}(g(x,u)^{-1}\p_sg(x,u))v.
\end{multline}

We recall that if $Y$ is a minimal surface (in the variational sense), then $t = 0$ is a critical point of the map $t \mapsto \text{Vol}(u + tv)$ for all functions $v$ that vanish on the boundary. We also denote 
\[
 f(u,\nabla u)=\frac{1}{2}\frac{1}{(1+\abs{\nabla u}^2_{g_u})^{1/2}}(\p_sg_u^{-1})(\nabla u,\nabla u)+\frac{1}{2}(1+\abs{\nabla u}^2_{g_u})^{1/2}\text{Tr}(g_u^{-1}\p_sg_u),
\]
where we shorthanded $g_u(x)=g(x,u(x))$ and $\p_sg_u=\p_sg_s|_{s=u}$ etc.
It follows by integrating by parts that the minimal surface equation is
\begin{equation*}
 -\frac{1}{\det(g_u)^{1/2}}\nabla\cdot \left( g_u^{-1}\frac{\det(g_u)^{1/2}}{\sqrt{1+\abs{\nabla u}^2_{g_u}}} \right)\nabla u + f(u,\nabla u) =0. 
\end{equation*}
This is \eqref{eq:minimal_surface_general}. Here we used that the boundary term arising from integrating by parts in \eqref{eq:first_var} is
\begin{equation}\label{eq:boundary_term}
 \int_{\p \Sigma} \frac{v}{\sqrt{1+\abs{\nabla u}^2_{g(x,u)}}}(\nabla u, \nu)_{g(x,u)}dS_{g(x,u)},
\end{equation}
which is zero since $v|_{\p \Sigma}=0$. Here $(\ccdot,\ccdot)_{g(x,u)}$ and $dS_{g(x,u)}$ denote the  inner product and the volume form on the $\p\Sigma$ induced by the Riemannian metric $g(x,u)$ respectively..

\subsection{Variation of area and the exterior problem}\label{sec:areas_exterior_problem}
Let us then explain how the DN map is related to areas of minimal surfaces. Here a technical issue arises: If a surface $\Sigma$ is embedded in a Riemannian manifold $N$, it might be that there is no positive $\eps >0$ such that the domain $\Sigma \times (-\eps,\eps)$ of Fermi coordinates is contained in $N$. It can also be that the boundary $(\partial \Sigma) \times ((-\eps,\eps)\setminus \{0\})$ intersects the interior of the unknown manifold $N$. Therefore, considering minimal surfaces as solutions to the
minimal surface equation \eqref{eq:minimal_surface_general} in Fermi-coordinates is technically complicated. To
address this technical issues, we consider the situation in an extended manifold $\widetilde N$ that contains $N$. We always assume $\widetilde N$ to be geodesically complete. 

We explain the relation of areas of minimal surfaces and the DN map using $\widetilde N$. This section is mainly for motivational purposes and we choose to keep the exposition short. We consider the following problem:
%
%
%
%

\noindent\textbf{The exterior problem:} Let $(\Sigma,g)$ be a minimal surface. Do $\p \Sigma$ and the volumes of minimal surfaces $\Sigma'$ in the exterior manifold $\widetilde N$, whose boundaries satisfy
$\partial \Sigma'\subset \widetilde N\setminus N$, determine the isometry type of $(\Sigma,g)$?


We say that a minimal surface $\Sigma$ extends properly to a minimal surface $\widetilde \Sigma$ if $\widetilde \Sigma$ is a minimal surface in $\widetilde N$, $\Sigma\subset \subset \widetilde \Sigma$ and $0$ is not an eigenvalue of the first linearization of the minimal surface equation \eqref{eq:minimal_surface_general} on $\widetilde \Sigma$. 
Let  $\delta_0>0$.
We say that $\mathcal S$   is  a family of smooth  minimal surface  deformations of the surface $\widetilde \Sigma$ when $\mathcal S$  consists of functions  $\mathcal F:(-\delta_0,\delta_0)\times  \widetilde \Sigma\to \widetilde N$ such that $\mathcal F(s,x)$, $(s,x)\in (-\delta_0,\delta_0)\times \widetilde \Sigma$
is in $C^\infty((-\delta_0,\delta_0);C^3( \widetilde \Sigma))$ and all surfaces $\widetilde \Sigma(s)=\{\mathcal F(s,x):\ x\in \widetilde \Sigma\}$ are minimal surfaces in $\overline N$.

For the exterior problem, we record the following corollary of Theorem \ref{thm:main}. 
In the proof of the corollary we are going to refer to Lemma \ref{lem:F_id_infty}, which can be found from Section \ref{sec:proof_of_main_thm}. 

\begin{corollary}[Uniqueness result for the exterior problem]\label{cor:exterior_problem_repeat}
Let $(\widetilde N,\overline g)$ and $(N, \overline g|_N)$ be $3$-dimensional Riemannian manifolds with boundaries such that $N \subset\subset \widetilde N$.
Let $(\Sigma,g)$ be a $2$-dimensional minimal surface and assume it extends properly to a minimal surface $\widetilde \Sigma \subset \widetilde N$. Assume also that $ \widetilde \Sigma$ is topologically equivalent to a domain $\Omega\subset \R^2$ by a diffeomorphism $ \widetilde \phi: \widetilde \Sigma\to \Omega$. 
Let $\mathcal S$ be  is  a family of smooth  minimal surface  deformations of the surface $\widetilde \Sigma$.

Then, the knowledge of $\widetilde N\setminus N$, $\overline g|_{\widetilde N\setminus N}$, $\widetilde \Sigma\setminus \Sigma$, $\phi|_{\widetilde \Sigma\setminus \Sigma}$ and the functions $s\to (\partial \widetilde \Sigma(s),
\hbox{Vol}\s (\widetilde \Sigma(s))$, $s\in (-\delta_0,\delta_0)$, where  $\widetilde \Sigma(s)=\{\mathcal F(s,x):\s x\in \widetilde \Sigma\}$ and
$\mathcal F \in \mathcal S$,
determine 
  $(\Sigma,g)$ up to a boundary preserving isometry. The isometry also preserves the first fundamental form.
\end{corollary}
We also have an analogous corollary to Theorem \ref{thm:main2}. 
%
To prove Corollary \ref{cor:exterior_problem}, we record in the next lemma how volumes of minimal surfaces are related to DN map in a domain of Fermi coordinates.

\begin{lemma} \label{lem:minimal surfaces and the DN map}
Let $(\Sigma, g)$ be a Riemannian manifold with a boundary $\p \Sigma$ and $ N= \Sigma\times (-\delta,\delta)$, $\delta >0$.  Assume that $N$ is equipped with a Riemannian metric of the form $ds^2 +g_{ab}(x,s)dx^adx^b$ and that $\epsilon >0$ is so small that for all $h$ in the set
$\mathcal W_\epsilon=\{h\in C^\infty(\p \Sigma):\|h\|_{C^{2,\alpha}(\p  \Sigma)}<\eps\}$
the minimal surface equation \eqref{eq:minimal_surface_general} has a unique solution $u_h: \Sigma\to \R$ such that $\norm{u_h}_{C^{2,\alpha}(\Sigma)}\leq \eps$, and $\epsilon<\delta$. 

Let $Y(h)=\{(u_h(x),x):\ x\in \Sigma\}$ be the minimal surface with the boundary value $h$.
Then the boundary $\p \Sigma$, the metric
$\overline g|_{\p  \Sigma\times (-\delta,\delta)}$ and 
the volumes $ \text{Vol}\s (Y(h))$ of the minimal surfaces, given  for all
$h\in \mathcal W_\epsilon$, determine the values of Dirichlet-to-Neumann map for the equation \eqref{eq:minimal_surface_general}, that is, $\Lambda_{\overline g}(h)$
 for 
$h\in \mathcal W_\epsilon$. Moreover, Fr\'echet derivatives of the map $h\to  \text{Vol}\s (Y(h))$ 
to the order $k+1$ at $h=0$  determine the  Fr\'echet derivatives of the map $h\to \Lambda_{\overline g}(h)$
at $h=0$ to the order $k$.
%
%
\end{lemma}
\begin{proof}

Let us define a non-linear boundary map $\mathcal{N}_g$ by
\begin{eqnarray*}
\mathcal{N}_{\overline g} (u|_{\p \Sigma})= \frac 1{\sqrt{1+\abs{\nabla u}^2_{g(x,u)}}}
 (\nu, \nabla u)_{g(x,u)}\bigg|_{\p \Sigma},
\end{eqnarray*}
where $u$ is the solution of the minimal surface equation \eqref{eq:minimal_surface_general}. Let $u_h$ be the solution of the minimal surface equation \eqref{eq:minimal_surface_general} with
boundary value  $u_h|_{\p\Sigma}=h$. By the calculations leading to \eqref{eq:boundary_term}, we see that 
%
the  volumes $ \text{Vol}(Y(h))$ of the minimal surfaces $Y(h)$
and their minimal surface variations $ \text{Vol}(Y(h+tw))$ determine
\begin{eqnarray*}
 \frac{d}{dt}\Big|_{t=0}\text{Vol}(Y(h+tw))
 &=&  \int_{\p \Sigma} \frac{w}{\sqrt{1+\abs{\nabla u}^2_{g(x,u)}}}(\nabla u, \nu)_{g(x,u)}dS_{g(x,u)}.
\end{eqnarray*}
Since the boundary $\p \Sigma$ and the metric $ \overline g|_{\p \Sigma\times (-\delta,\delta)}$ are known, and as $w\in C^\infty(\p \Sigma)$ is arbitrary, we see that $ \frac{d}{dt}\big|_{t=0}\text{Vol}(Y(h+tw))
$ with varying values of $w$ determine $\mathcal{N}_{\overline g} (h)$ and consequently also $\Lambda_{\overline g}$. Moreover, the above formula yields that the Fr\'echet derivatives of the map $h\to  \text{Vol}\s (Y(h))$ 
to the order $k+1$ at $h=0$  determine the  Fr\'echet derivatives of the map $h\to \Lambda_{\overline g}(h)$
at $h=0$ to the order $k$.
\end{proof}

\begin{proof}[Proof of Corollary \ref{cor:exterior_problem}]
Since $\widetilde N$ is complete, there exists a neighbourhood $\widetilde U$ of $\widetilde\Sigma$ in $\widetilde N$,
such that in  $\widetilde U$ there are  Fermi coordinates associated to $\widetilde\Sigma$. That is, there exists the Fermi coordinate map $\tilde \psi:\widetilde U\to \mathcal{D}:=\widetilde\Sigma \times (-\eps,\eps)$, for some $\eps>0$.

 Let $\mathcal S$ be   a family of smooth  minimal surface  deformations of the surface $\widetilde \Sigma$.
Let us consider a sub-family $\mathcal S_0$ of $\mathcal S$ that consists of functions 
$\mathcal F \in \mathcal S$ such that the surfaces  $\widetilde \Sigma(s)=\{\mathcal F(s,x):\ x\in \widetilde \Sigma\}$ satisfy $\mathcal F(s,x) \in \tilde \psi^{-1}(\p \widetilde \Sigma \times (-\eps,\eps))$ for 
all $x\in \p \widetilde \Sigma$  and $s\in (-\delta_0,\delta_0)$.
Let us next consider a function $\mathcal F \in \mathcal S_0$ and the minimal 
surfaces $\widetilde \Sigma(s)$  determined by $\mathcal F$. Then there is $\delta_1<\delta_0$  such that 
for $|s|<\delta_1$  the functions
$h_s=\mathcal F(s,\ccdot)|_{\p  \widetilde \Sigma}$  satisfy 
$\|h_s\|_{C^{2,\alpha}(\p \widetilde  \Sigma)}<\eps$, where $\eps$  is such that the minimal surface equation has
 a unique small solution with the Dirichlet boundary value $h_s$, see Proposition \ref{prop:local_well_posedness}. Moreover, by  Lemma \ref{lem:minimal surfaces and the DN map},
the functions  $s\to h_s$  and $s\to 
\hbox{Vol}(\widetilde \Sigma(s))$, $s\in (-\delta_1,\delta_1)$ determine the
 derivatives of the map $s\to \Lambda_{\overline g}(h_s)$ 
at $h=0$ to any order $k$.
 By the proof of Theorem \ref{thm:main}, these data determines $(\widetilde \Sigma, g)$ up to an isometry that preserves $\p \widetilde \Sigma$. The isometry also preserves the second scalar fundamental form.

 What is left is to show that the isometry restricts to an isometry of $\Sigma$. For this we use a uniqueness result for isometries. It follows from the proof of Theorem \ref{thm:main} and Lemma \ref{lem:F_id_infty} that the isometry agrees with the identity map of $\widetilde \Sigma \setminus \Sigma$ to infinite order on $\p\widetilde\Sigma$. Since by assumption $\widetilde \Sigma\setminus \Sigma$ and $\overline g|_{\widetilde N\setminus N}$ are known, also the induced metric on $\widetilde \Sigma\setminus \Sigma$ is known. By the proof of \cite[Theorem 3.3]{lassas2019conformal}, isometries agreeing to high order order on an open subset of boundary are unique. (For this, apply the proof of \cite[Theorem 3.3]{lassas2019conformal} with harmonic coordinates instead of conformal harmonic coordinates.) By the above facts, it follows that the isometry restricted to $\widetilde \Sigma \setminus \Sigma$ is unique and consequently the identity.  
 By injectivity it then follows that the isometry maps $\Sigma$ onto itself. By continuity the isometry is the identity on $\p \Sigma$. This proves the claim.  
\end{proof}

\subsection{Higher order linearization}\label{ss2.4}

In this section we discuss the higher order linearization method for the minimal surface equation on $(\Sigma,g)$. While we assume in this paper that $\Sigma$ is $2$-dimensional, the computations in this section hold in higher dimensions as well. We will derive the corresponding integral identities for the first, second and third order linearizations. 

Later we will see that the first order linearization can be used to determine a $2$-dimensional minimal surface $\Sigma$ up to a conformal transformation.  Second order linearization will used to recover $\p_s|_{s=0}g(x,s)$ in Fermi coordinates of $\Sigma$. The third order linearization, together what can be recovered by considering the first and second linearizations, will be used to recover the related conformal factor. 

For $j=1,\ldots,4$ let $\eps_j\in \R$ and $f_j\in C^{2,\alpha}(\p M)$ for some $0<\alpha<1$. Let us denote $\eps=(\eps_1,\eps_2,\eps_3,\eps_4)$. We consider boundary values $f=f_\eps$  of the form 
\begin{align}\label{f_epsilon}
	f_\eps:=\displaystyle \sum_{j=1}^4\eps_j f_j
\end{align}
for the minimal surface equation
\begin{equation}\label{eq:minimal_surf_sec2}
\begin{aligned}
\begin{cases}
 -\frac{1}{\det(g_u)^{1/2}}\nabla\cdot \left( g_u^{-1}\frac{\det(g_u)^{1/2}}{\sqrt{1+\abs{\nabla u}^2_{g_u}}} \right)\nabla u + f(u,\nabla u) =0, 
&\text{ in } \Sigma,
\\
u= f_\eps 
&\text{ on }\p \Sigma,
\end{cases}
    \end{aligned}
\end{equation}
in Fermi coordinates associated to $\Sigma\subset N$. Recall that in Fermi-coordinates $\overline g=ds^2 +g_{ab}(x,s)dx^adx^b$. 
Observe that $f_\eps \in U_\delta$ for sufficiently small $\epsilon$, so that by Proposition \ref{prop:local_well_posedness} the problem \eqref{eq:minimal_surf_sec2} is well-posed.

In this paper, we use the positive sign convention for the Laplacian. In local coordinates of $(\Sigma,g)$
\[
 \Delta_g u=-\nabla\cdot\nabla u= -\abs{g}^{1/2}\p_a\big(\abs{g}^{1/2}g^{ab}\p_b u\big).
\]
We denote by $\nabla$, $\Delta$, $\ccdot$ and $\abs{\ccdot}$ the corresponding covariant derivative, Laplacian, inner product and norm given be the metric $g$ if there is no change of confusion. We record the higher order linearizations of \eqref{eq:minimal_surf_sec2} at $\eps=0$, which corresponds to zero solution. We denote the first, second and third linearizations by
\begin{equation}\label{eq:not_for_lins}
 v^{j}:= \left.\frac{\p}{\p\epsilon_j}\right|_{\eps=0} u_f, \quad w^{jk}:= \left.\frac{\p^2}{\p\epsilon_j\p\epsilon_k}\right|_{\eps=0} u_f, \quad w^{jkl}:= \left.\frac{\p^3}{\p\epsilon_j\p\epsilon_k\p\epsilon_l}\right|_{\eps=0} u_f
\end{equation}
respectively. 

Next we compute the linearized equations that $v^{j}$, $w^{jk}$ and $w^{jkl}$ solve. For this, we set up some notation. 
Let us calculate the higher order linearization of the minimal surface equation. For this, for $l=1,2,\ldots$, we set up some notation. We denote
\begin{align*}
 d_u=\abs{g_u}^{1/2}, \quad h_u=\text{Tr}(g_u^{-1}\p_sg_u) \ \text{ and } \ k_u&=g_u^{-1}.
 \end{align*}
and $d=\abs{g_u}^{1/2}|_{u=0}$ and $d^{(1)}:=\p_u|_{u=0}\abs{g_u}^{1/2}$. 
Note that 
\begin{equation}\label{eq:first_der_of_d_zero}
 d^{(1)}=\Big(\frac{1}{2}\abs{g_u}^{1/2}\text{Tr}(g_u^{-1}\p_sg_u)\Big)\Big|_{u=0}=\frac{1}{2}dh=0.
\end{equation}
We also denote
\begin{align*}
h&=h_u|_{u=0},  \quad h^{(l)}_u=\left(\frac{\p}{\p u}\right)^lh_u, \quad h^{(l)}=\left(\frac{\p}{\p u}\right)^l\Big|_{u=0}h_u, \\
 k&=k_u|_{u=0}, \quad  k^{(l)}_u=\left(\frac{\p}{\p u}\right)^lk_u, \quad k^{(l)}=\left(\frac{\p}{\p u}\right)^l\Big|_{u=0}k_u 
\end{align*}
for $l=1,2,3$. 
Note that $k_u,k,k^{(l)}_u, k^{(l)}$ are symmetric $2$-tensor fields on $\Sigma$. 
With these notations, the minimal surface equation \eqref{eq:minimal_surface_general} can be written in local coordinates on $\Sigma$ as 
\begin{multline*}
 0=-\frac{1}{\det(g_u)^{1/2}}\nabla\cdot \left( g_u^{-1}\frac{\det(g_u)^{1/2}}{\sqrt{1+\abs{\nabla u}^2_{g_u}}} \right)\nabla u + f(u,\nabla u)  \\
 =-d_u^{-1}\s \nabla\cdot\big(k_ud_u(1+\abs{\nabla u}^2_{g_u})^{-1/2}\nabla u\big)+f(u,\nabla u), 
\end{multline*}
where
\[
 f(u,\nabla u)=\frac{1}{2}\frac{1}{(1+\abs{\nabla u}^2_{g_u})^{1/2}}k_u^{(1)}(\nabla u,\nabla u)+\frac{1}{2}(1+\abs{\nabla u}^2_{g_u})^{1/2}h_u. 
\]
Here $\nabla$ is defined with respect to $\R^n$ metric. 
Since $u\equiv 0$ is a solution, we have $f(0,0)=0$, which implies 
\begin{equation}\label{eq_h_equiv_0}
h=\text{Tr}(g^{-1}\p_s|_{s=0}g)=0. 
\end{equation}

Next we write the minimal surface equation as 
\begin{equation*}
 0=
 P^uu+f(u,\nabla u ), 
\end{equation*}
where $P^u$ is the partial differential operator given by
\[
 P^uF=-\frac{1}{\det(g_u)^{1/2}}\nabla\cdot \left( g_u^{-1}\frac{\det(g_u)^{1/2}}{\sqrt{1+\abs{\nabla u}^2_{g_u}}} \right)\nabla F. 
\]
We then have
\[
 P:=P^{u}|_{u=0}=\Delta_g.
\]
We will see that factoring the minimal surface equation in this way is beneficial for calculations.  We also denote
\begin{align*}
 &P_{j}^u=\left(\frac{\p}{\p \eps_j}P^u\right), && P_{jk}^u=\left(\frac{\p^2}{\p \eps_k \p \eps_j}P^u\right),  &&& P_{jkl}^u=\left(\frac{\p^3}{\p \eps_k \p \eps_j\p\eps_l}P^u\right) \\ 
 &P^j=\left(\frac{\p}{\p \eps_j}P^u\right)\Big|_{\eps=0}, && P^{jk}=\left(\frac{\p^2}{\p \eps_k \p \eps_j}P^u\right)\Big|_{\eps=0},  &&& P^{jkl}=\left(\frac{\p^3}{\p \eps_k \p \eps_j\p\eps_l}P^u\right)\Big|_{\eps=0} \\
 &u_j=\left(\frac{\p}{\p \eps_j}u\right), && u_{jk}=\left(\frac{\p^2}{\p \eps_k \p \eps_j}u\right), &&& u_{jkl}=\left(\frac{\p^3}{\p \eps_k \p \eps_j\p \eps_l}u\right).
\end{align*}
With this notation we have
\[
 v^{j}=u_j|_{u=0}, \quad w^{jl}=u_{jl}|_{u=0}, \quad w^{jkl}=u_{jkl}|_{u=0}
\]
We will also use for convenience the physicists' short hand notation where indices without a specified value in brackets are symmetrised over. For example
\[
 P^u_{(j}u_{k)}=P^u_{j}u_{k}+P^u_{k}u_{j}
\]
and
\[
 P^u_{(jl}u_{k)}=P^u_{jl}u_{k}+P^u_{jk}u_{l}+P^u_{kl}u_{j}.
\]

With the above notations, the equation for the first linearization is
\[
 0=P_j^uu+P^uu_j+\p_{\eps_j}f.
\]
The equation for the second linearization is
\[
 0=P_{jk}^uu+P_{j}^uu_{k} +P_k^uu_j +P^u u_{jk}+\p_{\eps_{jk}}f =P^u_{jk}u+P^u_{(j}u_{k)}+P^u u_{jk}+\p_{\eps_{jk}}f,
 \]
and for the third linearization it is 
\begin{multline*}
 0=P_{jkl}^uu+P_{jk}^uu_l+P_{jl}^uu_{k}+P_{j}^uu_{kl}+P_{kl}^uu_{j}+P_k^uu_{jl}+P^u_{l}u_{jk}+P^u u_{jkl}+\p_{\eps_{jkl}}f\\
 =P^u_{jkl}u+P^u_{(jl}u_{k)}+P^u_{(j}u_{kl)}+P^u u_{jkl}+\p_{\eps_{jkl}}f.
\end{multline*}

Note that 
\[
P_j^uu|_{u=0}=P_{jk}^uu|_{u=0}=P_{jkl}^uu|_{u=0}=0.
\]
Thus evaluating the linearizations at $u=0$, equivalently at $\eps=(\eps_1,\eps_2,\eps_3)=0$, we obtain the equation for the first
\[
  0=Pv_j+\p_{\eps_j}|_{\eps=0}f=(\Delta_g+h^{(1)}/2)v_j 
\]
second
\[
  0=P^{(j}v^{k)}+P w^{jk}+\p_{\eps_{jl}}|_{\eps=0}f=(\Delta_g+h^{(1)}/2)w^{jk}+P^{(j}v^{k)}+k^{(1)}(\nabla v^j,\nabla v^k)+\frac{1}{2}h^{(2)}v^jv^k
\]
and third
\begin{multline*}
  0=P^{(jl}v^{k)}+P^{(j}w^{kl)}+P w^{jkl}+\p_{\eps_{jkl}}|_{\eps=0}f \\
  =(\Delta_g+h^{(1)}/2)w^{jkl}+P^{(jl}v^{k)}+P^{(j}w^{kl)}+k^{(2)}(\nabla v^{(j},\nabla v^k)v^{l)}+k^{(1)}(\nabla v^{(j},\nabla w^{kl)}) \\
  +\frac{1}{2}g(\nabla v^{(j},\nabla v^k)v^{l)}h^{(1)} +\frac{1}{2}w^{(jk}v^{l)}h^{(2)}+\frac{1}{2}v^{j}v^kv^{l}h^{(3)} 
\end{multline*}
linearizations at $\eps=0$. 

Here we used that
\begin{equation*}
 \p_{\eps_j}|_{\eps=0}f=(1+\abs{\nabla u}^2_{u})^{1/2}h_u=\frac{1}{2}\p_{\eps_j}|_{\eps=0}h_u=\frac{1}{2}h^{(1)}v^j,
\end{equation*}
since $h=0$ by \eqref{eq_h_equiv_0}. Similarly, we used
\begin{equation*}
 \p_{\eps_{jl}}|_{\eps=0}f=\frac{1}{2}\p_{\eps_{jl}}|_{\eps=0}h_u=\frac{1}{2}h^{(2)}v^jv^l.
\end{equation*}
We also used 
\begin{multline}\label{eq:third_deriv_of_f}
 \p_{\eps_{jkl}}|_{\eps=0}f
=k^{(2)}(\nabla v^{(j},\nabla v^k)v^{l)}+k^{(1)}(\nabla w^{(jk},\nabla v^{l)})+\frac{1}{2}g(\nabla v^{(j},\nabla v^k)v^{l)}h^{(1)} \\
 +\frac{1}{2}w^{(jk}v^{l)}h^{(2)}+\frac{1}{2}v^{j}v^{k}v^{l}h^{(3)}+\frac{1}{2}w^{jkl}h^{(1)}.
\end{multline}
We have placed the calculation behind \eqref{eq:third_deriv_of_f} in Appendix \ref{appx:calculations}.

By collecting the results of the above calculations, we obtain:
\begin{lemma}[Higher order linearizations]\label{lem:high_ord_lin} Let $f$ be as in \eqref{f_epsilon}, and for $j,k,l\in \{1,\ldots,4\}$ let $v^{j}$, $w^{jk}$ and $w^{jkl}$ be as in \eqref{eq:not_for_lins}.

 \noindent\textbf{(1)} The first linearization $v^{j}$ 
 satisfies the equation 
 \begin{equation}\label{linear_eq}
	\begin{aligned}
		\begin{cases}
			(\Delta_g+h^{(1)}/2)v^j=0 
			& \text{ in } \Sigma,
			\\
			v^{j}=f_j
			&\text{ on }\p \Sigma.
		\end{cases}
	\end{aligned}
\end{equation}
 \noindent\textbf{(2)} The second linearization $w^{jk}$ 
 satisfies 
 \begin{equation}\label{2nd_lin_eq}
 \begin{aligned}
		\begin{cases}
  (\Delta_g+h^{(1)}/2)w^{jk}+P^{(j}v^{k)}+k^{(1)}(\nabla v^{j},\nabla v^{k})+\frac{1}{2}h^{(2)}v^{j}v^{k}= 0& \text{ in } \Sigma \\
  w^{jk}=0
			&\text{ on }\p \Sigma.
  		\end{cases}
	\end{aligned}
 \end{equation}

 \noindent\textbf{(3)} The third linearization $w^{jkl}$ 
 satisfies the equation
  \begin{equation}\label{3rd_lin_eq}
	\begin{aligned}
		\begin{cases}
			(\Delta_g+h^{(1)}/2)w^{jkl}+P^{(jl}v^{k)}+P^{(j}w^{kl)} \\
	\qquad	+k^{(2)}(\nabla v^{(j},\nabla v^{k})v^{l)}+k^{(1)}(\nabla v^{(j},\nabla w^{kl)}) \\
  \qquad\quad+\frac{1}{2}g(\nabla v^{(j},\nabla v^{k})v^{l)}h^{(1)} +\frac{1}{2}w^{(jk}v^{l)}h^{(2)}+\frac{1}{2}v^{j}v^{k}v^{l}h^{(3)}=0 
			& \text{ in } \Sigma. \\
			w^{jkl}=0
			&\text{ on }\p \Sigma,
		\end{cases}
	\end{aligned}
\end{equation}
\end{lemma}

\subsection{Integral identities} \label{ss2.5}
Next we derive the corresponding integral identities for the second and third linearized equations.
		Let $v^m$ be solution to the first linearization
\[
 (\Delta_g+h^{(1)}/2)v^m=0
\]
with boundary value $f_m$.  Let us denote
 \[
  I_2=P^{(j}v^{k)}+k^{(1)}(\nabla v^{j},\nabla v^{k})+\frac{1}{2}h^{(2)}v^{j}v^{k}.
 \]
 Then, by \eqref{2nd_lin_eq} the second linearization $w^{jk}$ solves
 \[
  (\Delta_g+h^{(1)}/2)w^{jk}=-I_2.
 \]
Integration by parts yields
\begin{multline}\label{eq:2nd_lin_calc}
 \int_{\p \Sigma}f_m\p_\nu w^{jk}dS=\int_{\Sigma} v^m\Delta_{g} w^{jk}dV + \int_{\Sigma} \nabla v^m\cdot \nabla w^{jk}dV \\
 =-\frac{1}{2}\int_{\Sigma} v^mh^{(1)} w^{jk}dV-\int_\Sigma v^mI_2dV 
 + \int_{\p \Sigma}w^{jk}\p_\nu v^mdS- \int_{\Sigma}w^{jk} \Delta_{g} v^mdV \\
 =-\int_\Sigma v^m I_2dV. 
\end{multline}
In the last equality we used $v^m$ solves the first linearized equation, which canceled the integrals involving $w^{jk}$ over $\Sigma$. We also used that $w^{jk}|_{\p \Sigma}=0$.

Next we calculate  $-\int_\Sigma v^mI_2dV $. We have
\begin{equation*}
 P^j=-\frac{d}{d\eps_j}\Big|_{\eps=0}d_u^{-1}\nabla\cdot k_ud_u(1+\abs{\nabla u}^2_u)^{-1/2}\nabla =-d^{-1}\nabla\cdot k^{(1)}v^jd\nabla.
\end{equation*}
Here we used $d^{(1)}=0$. Using the $d^{-1}\nabla\cdot k d$ is the Riemannian divergence $\nabla^g\ccdot$ on $(\Sigma,g)$, we may write the operator $P^{j}$ as
\begin{equation}\label{eq:formula_for_Pj}
 P^j=-\nabla^g\cdot (v^jgk^{(1)}\nabla).
\end{equation}
%
Thus 
\begin{equation*}
 \int_\Sigma v^m P^jv^kdV
 =\int_\Sigma v^jk^{(1)}(\nabla v^m,\nabla v^k)dV - \int_{\p \Sigma} v^mv^j k^{(1)}(\nu,\nabla v^k)dS
\end{equation*}
and it follows that 
\begin{equation*}
 \int_\Sigma v^m P^{(j}v^{k)}dV
 =\int_\Sigma k^{(1)}(\nabla v^m,\nabla v^{(j})v^{k)}dV \\
 - \int_{\p \Sigma} v^m k^{(1)}(\nu,\nabla v^{(j})v^{k)}dS. 
\end{equation*}
By collecting the results of above calculations, we have proven:
\begin{lemma}[Integral identity for the second linearization]\label{Lem:Integral identity_2nd} Let $f_\eps$ be as in \eqref{f_epsilon}, and for $j,k,m\in \{1,2,3\}$ let $v^{j}$ and $w^{jk}$ be as in \eqref{eq:not_for_lins}.
	The integral identity for the second linearization is
	\begin{multline}\label{eq:second_integral_id}
	 \int_{\p \Sigma} f_m \s \p^2_{\eps_j \eps_k}\big|_{\epsilon=0} \Lambda (f_\epsilon) \, dS_g =\int_{\Sigma} v^m k^{(1)}(\nabla v^k,\nabla v^j\big)dV+\int_{\Sigma}v^k k^{(1)}(\nabla v^j,\nabla v^m\big)dV\\
 +\int_{\Sigma} v^j k^{(1)}(\nabla v^k,\nabla v^m\big)dV -\frac{1}{2}\int_{\Sigma} h^{(2)}v^jv^kv^mdV  \\
 - \int_{\p \Sigma} v^m k^{(1)}(\nu,\nabla v^{(j})v^{k)}dS,
	\end{multline}
	which holds for any $j,k,m\in \{1,2,3\}$.
\end{lemma}

Next we turn to deriving the integral identity for the third linearization, which in full generality will be pretty complicated. We first give the identity and then derive it. For this, we denote
\begin{equation}\label{eq:HRB}
 \begin{split}
 H&=-\int_{\Sigma}v^{(j}v^{k} k^{(2)}(\nabla v^{l)},\nabla v^m) dV +\int_{\Sigma}v^m g(\nabla (d^{-1}d^{(2)}v^{(j}v^k),\nabla v^{l)})dV \\
 &\qquad-\int_\Sigma v^m k^{(2)}(\nabla v^{(j},\nabla v^{k})v^{l)}dV-\frac{1}{2}\int_\Sigma v^mv^{j}v^{k}v^{l}h^{(3)}dV,\\
 R&=-\int_{\Sigma}w^{(jk}k^{(1)}(\nabla v^{l)},\nabla v^m)dV-\int_\Sigma k^{(1)}(\nabla v^m,\nabla v^{(j})w^{kl)}dV \\
 &\qquad-\int_{\Sigma}v^mk^{(1)}(\nabla v^{(j},\nabla w^{kl)})dV-\frac{1}{2}\int_\Sigma v^mg(\nabla v^{(j},\nabla v^{k})v^{l)}h^{(1)}dV  \\
 &\qquad \qquad -\frac{1}{2}\int_\Sigma v^mw^{(jk}v^{l)}h^{(2)}dV,   \\
 B&=\int_{\p \Sigma} v^mv^{(j}v^k g(\nu,k^{(2)}\nabla v^{l)})dS+\int_{\p \Sigma} v^mw^{(jk} g(\nu,k^{(1)}\nabla v^{l)})dS\\
&\qquad-\int_{\p \Sigma} v^mg(\nabla v^{(j},\nabla v^k)\p_\nu v^{l)}dS+\int_{\p \Sigma} v^mv^{(j} k^{(1)}(\nu,\nabla w^{kl)})dS.
\end{split}
\end{equation}

The integral identity for the third linearization then is:
\begin{lemma}[Integral identity for the third linearization]\label{Lem:Integral identity_3rd} Let $f_\eps$ be as in \eqref{f_epsilon}, and for $j,k,l,m\in \{1,\ldots,4\}$ let $v^{j}$, $w^{jk}$ and $w^{wjk}$ be as in \eqref{eq:not_for_lins}.

	The integral identity for the third linearization is
	\begin{multline}\label{eq:third_integral_id}
	 \int_{\p \Sigma} f_m \s \p^3_{\eps_j \eps_k\eps_l}\big|_{\epsilon=0} \Lambda (f_\epsilon) \, dS_g = \int_{\Sigma}g(\nabla v^{j},\nabla v^k)g(\nabla v^{l},\nabla v^m) dV  \\
	 +\int_{\Sigma}g(\nabla v^{j},\nabla v^l)g(\nabla v^{k},\nabla v^m) dV+\int_{\Sigma}g(\nabla v^{l},\nabla v^k)g(\nabla v^{j},\nabla v^m) dV \\
	 +H+R+B,
	\end{multline}
	which holds for any $j,k,m\in \{1,\ldots,4\}$. Here $H$, $R$ and $B$ are as in \eqref{eq:HRB}.
\end{lemma}
As we see, the integral identity in full generality is be pretty complicated. 
When we apply the integral identity in the proof of Theorem \ref{thm:main} two things will however happen: Firstly, the terms of $H$ will be of lower order when we use complex geometrical optics (CGOs) as the functions $v^j$. Secondly, the terms of $R$ will be recovered from the second linearization. Thus we will be able to neglect both the $H$ and $R$ terms in the proof. We will also be able to neglect the terms of $B$ as we assume boundary determination. We also note that the first three terms on the right hand side of \eqref{eq:third_integral_id} are not conformally invariant in conformal scalings. This fact will be used to recover a conformal factor in the proof.

The derivation of the integral identity is a straightforward but long calculation. A reader uninterested of this may jump directly to Section \ref{sec:proof_of_main_thm}. 
This time we denote
\begin{multline}\label{eq:I3}
 I_3=P^{(jl}v^{k)}+P^{(j}w^{kl)} +k^{(2)}(\nabla v^{(j},\nabla v^{k})v^{l)}+k^{(1)}(\nabla v^{(j},\nabla w^{kl)}) \\
  \qquad\quad+\frac{1}{2}g(\nabla v^{(j},\nabla v^{k})v^{l)}h^{(1)} +\frac{1}{2}w^{(jk}v^{l)}h^{(2)}+\frac{1}{2}v^{j}v^{k}v^{l}h^{(3)}
\end{multline}
so that $w^{jkl}$ solves
\[
 (\Delta_g+h^{(1)}/2)w^{jkl}=-I_3.
\]
By the same calculation as in \eqref{eq:2nd_lin_calc}, we obtain
\begin{equation}\label{eq:3rd_lin_calc}
 \int_{\p \Sigma}f_m\p_\nu w^{jkl}dS=-\int_\Sigma v^m I_3dV.
\end{equation}
We already calculated the formula for $P^j$ in \eqref{eq:formula_for_Pj}. The formula for $P^{jk}$ is 
\begin{multline}\label{eq:formula_for_Pjk}
 P^{jk}=
-k\nabla (d^{-1}d^{(2)}v^k v^l)\cdot \nabla-d^{-1}\nabla\cdot (k^{(2)}v^kv^ld \nabla)-d^{-1}\nabla\cdot (k^{(1)}w^{kl}d \nabla)
 \\
 +d^{-1}\nabla\cdot (kd g(\nabla v^k,\nabla v^l)\nabla)
\end{multline}
We have placed the calculation how the above formula is derived in Appendix \ref{appx:calculations}. 
Using again that $d^{-1}\nabla\cdot k d$ is the Riemannian divergence, we may write the operator $P^{jk}$ as
\begin{multline*}
  P^{jk}=-g(\nabla (d^{-1}d^{(2)}v^k v^l), \nabla \ccdot)
  -\nabla^g\cdot (v^jv^k gk^{(2)}\nabla) 
 -\nabla^g\cdot (w^{jk}gk^{(1)} \nabla) \\
  +\nabla^g\cdot (g(\nabla v^j,\nabla v^k)\nabla). 
\end{multline*}
 By integration by parts, we then have
\begin{multline}\label{eq:integral_P_jk}
 -\int_\Sigma v^mP^{(jk}v^{l)}=\int_{\Sigma}g(\nabla v^{(j},\nabla v^k)g(\nabla v^{l)},\nabla v^m) dV-\int_{\Sigma}v^{(j}v^{k} k^{(2)}(\nabla v^{l)},\nabla v^m) dV \\
  +\int_{\Sigma}v^m g(\nabla (d^{-1}d^{(2)}v^{(j}v^k),\nabla v^{l)})dV-\int_{\Sigma}w^{(jk}k^{(1)}(\nabla v^{l)},\nabla v^m)dV + \int_{\p \Sigma}B_1dS,
\end{multline}
where
\begin{multline}\label{eq:defs_for_B2}
B_1=\int_{\p \Sigma} v^mv^{(j}v^k g(\nu,k^{(2)}\nabla v^{l)})dS+\int_{\p \Sigma} v^mw^{(jk} g(\nu,k^{(1)}\nabla v^{l)})dS\\
-\int_{\p \Sigma} v^mg(\nabla v^{(j},\nabla v^k)\p_\nu v^{l)}dS.
\end{multline}
Here $\nu$ is the normal vector field on $\p \Sigma$ with respect to the metric $g$. Here also for example $g(\nu,k^{(1)}\nabla v^{l})=g_{ab}\nu^a(k^{(1)})^{bc}\p_cv^{l})$.

By \eqref{eq:formula_for_Pj} and integration by parts we have
\begin{equation}\label{eq:formula_for_Pj_2}
 -\int_\Sigma v^m P^{(j}w^{kl)}dV
 =-\int_\Sigma k^{(1)}(\nabla v^m,\nabla v^{(j})w^{kl)}dV \\
 + \int_{\p \Sigma} v^mv^{(j} k^{(1)}(\nu,\nabla w^{kl)})dS.
\end{equation}
Recall that
\begin{multline}\label{eq:I3_2}
   \int_\Sigma v^m I_3dV= \int_\Sigma v^mP^{(jk}v^{l)}+\int_\Sigma v^m P^{(j}w^{kl)} \\
 	+\int_\Sigma v^mk^{(2)}(\nabla v^{(j},\nabla v^{k})v^{l)}+\int_\Sigma v^mk^{(1)}(\nabla v^{(j},\nabla w^{kl)}) \\
   \qquad\quad+\frac{1}{2}\int_\Sigma v^mg(\nabla v^{(j},\nabla v^{k})v^{l)}h^{(1)} +\frac{1}{2}\int_\Sigma v^mw^{(jk}v^{l)}h^{(2)}+\frac{1}{2}\int_\Sigma v^mv^{j}v^{k}v^{l}h^{(3)}.
\end{multline}
With the notations for $H$, $R$ and $B$ in \eqref{eq:HRB}, and by \eqref{eq:integral_P_jk}, \eqref{eq:formula_for_Pj_2}  and  \eqref{eq:I3_2}, we then have 
\[
 -\int_\Sigma v^m I_3dV=\int_{\Sigma}g(\nabla v^{(j},\nabla v^k)g(\nabla v^{l)},\nabla v^m) dV+\int_\Sigma H dV+\int_\Sigma R dV+\int_{\p \Sigma} B dS.
\]
By \eqref{eq:3rd_lin_calc} we had 
\begin{equation*}
 \int_{\p \Sigma}f_m\p_\nu w^{jkl}dS=-\int_\Sigma v^m I_3dV.
\end{equation*}
We have obtained the integral identity for the third linearization \eqref{eq:third_integral_id}.

\section{Complex geometrics optics solutions and their calculus}\label{sec:CGOs}
In this section we recall the construction of complex geometrics optics solutions (CGOs) from \cite{guillarmou2011identification} and develop a calculus for estimating their remainder terms. The latter is needed as the dependence of the remainders is less explicit in the small parameter $h>0$ than it is for main part of the CGOs.
\subsection{Construction of CGOs}

To begin, we assume that our Riemann surface $\Sigma$ is compactly contained in the open surface $M$ which in turn is compactly contained in the open surface $\widetilde M$ whose closure is a surface with boundary. 

For $q\in C^\infty_c(M)$, we recall the construction of \cite{guillarmou2011identification} of complex geometrics solutions to
\[
 (\Delta +q)u=0
\]
 on the Riemann surface $ M$. We keep the presentation brief and refer to \cite[Section 2]{guillarmou2011identification} for details. For a summary about the Riemannian differential calculus on Riemannian surfaces we refer to \cite{guillarmou2011calderon}. The complexified cotangent bundle $\C T^*\widetilde M$ has the splitting
\[
 \C T^*\widetilde M= T^*_{1,0}\widetilde M \oplus T^*_{0,1}\widetilde M
\]
determined by the eigenspaces of the Hodge star operator $\star$. 
In local  complex coordinate $z$ the space $T^*_{1,0}\widetilde M$ is spanned by $dz$ and $T^*_{0,1}\widetilde M$ is spanned by $d\ol z$.  The invariant definitions of $\p$ and $\op$ operators are given as
\[
 \op:=\pi_{0,1}d \text{ and } \p:=\pi_{1,0}d.
\]
By \cite[Proposition 2.1]{guillarmou2011identification} there is a right inverse $\op^{-1}$ for $\op$ in the sense that
\[
 \op\s \s  \op^{-1}\omega=\omega \text{ for all } \omega\in C_0^\infty(\widetilde M,T_{1,0}^*\widetilde M)
\]
such that $\op^{-1}$ is bounded from $L^p(T_{1,0}^*\widetilde M)$ to $W^{1,p}(\widetilde M)$ for any $p\in (1,\infty)$. We have analogous properties for 
\[
 \op^*=-i\star \p: W^{1,p}(T_{0,1}^*\widetilde M)\to L^p(\widetilde M),
\]
which is the Hermitean adjoint of $\op$. In complex coordinate $z$, the operator $\op^*$ is just $\p$.

We define 
\[
 \op_\psi^{-1}:=\mathcal{R}\op^{-1}e^{-2i\psi/h}\mathcal{E} \text{ and } \op_\psi^{*-1}:=\mathcal{R}\op^{*-1}e^{2i\psi/h}\mathcal{E},
\]
{\color{black} where $\mathcal{E} : W^{l,p}(M) \to W_c^{l,p}(\widetilde M)$ an extension operator for some $\widetilde M$ compactly containing $M$ and $\mathcal R$ is the restriction operator}.
By \cite[Lemma 2.2 and Lemma 2.3]{guillarmou2011identification} we have for $p>2$ and $2\leq q\leq p$ the following estimates 
\begin{align}\label{eq:sobo_decay}
\begin{split}
 \norm{\overline \p_\psi^{-1}\omega}_{L^q(M)}&\leq C h^{1/q}\norm{\omega}_{W^{1,p}(M, T_{0,1}^*M)} \\
 \norm{\overline \p_\psi^{*-1}\omega}_{L^q(M)}&\leq C h^{1/q}\norm{\omega}_{W^{1,p}(M, T_{1,0}^*M)}.
 \end{split}
\end{align}
Moreover, there is $\eps>0$ such that 
\begin{align}\label{eq:sobo_decayL2}
\begin{split}
 \norm{\overline \p_\psi^{-1}\omega}_{L^2(M)}&\leq C h^{1/2+\eps}\norm{\omega}_{W^{1,p}(M, T_{0,1}^*M)} \\
 \norm{\overline \p_\psi^{*-1}\omega}_{L^2(M)}&\leq C h^{1/2+\eps}\norm{\omega}_{W^{1,p}(M, T_{1,0}^*M)}.
 \end{split}
\end{align}

The complex geometrics optics solutions (CGOs) we use are the same as in \cite{guillarmou2011identification}, but our notation is slightly different. The CGOs are of the form 
\[
 v=e^{\Phi/h}(a+r_h),
\]
where $\Phi=\phi+i\psi$ is holomorphic Morse function and $a$ is a holomorphic function defined on $\widetilde M$, cf. \cite[Proposition 3.1, Eq. 21]{guillarmou2011identification}. 
In particular, $v$ solves $(\Delta + q)v=0$ if and only if
\begin{align}\label{eq:equation_for_rh}
\begin{split}
r_h=-\overline{\p}_\psi^{-1}s_h,\  (1+\op_\psi^{*-1}q\overline\p_\psi^{-1})s_h=\op_\psi^{*-1}(qa).
 \end{split}
\end{align}
To obtain an explicit expression for $s_h$, we introduce
\begin{eqnarray}\label{def: Th}
 T_h:=-\op_\psi^{*-1}q\overline\p_\psi^{-1}.
\end{eqnarray}
and note that its formal transpose is given by 
\begin{eqnarray}
\label{adjoint of T}
T_h^t = e^{-2i\psi/h} \p^{*-1} qe^{2i\psi/h} \p^{-1}
\end{eqnarray}
(We note that $T_h$ is not exactly the operator $S_h$ in \cite{guillarmou2011identification}. The reason for this is that we use holomorphic amplitude $a$, whereas the construction in \cite{guillarmou2011identification} uses antiholomorphic amplitude.)
By the proof of \cite[Lemma 3.1]{guillarmou2011identification}, we have
\[
 \norm{T_h}_{L^r\to L^r}=O(h^{1/r}) \text{ and } \norm{T_h}_{L^2\to L^2}=O(h^{1/2-\eps}), 
\]
for any $0<\eps<1/2$.

Thus we may derive an explicit expression for $s_h$ from \eqref{eq:equation_for_rh} via Neumann series as
\begin{eqnarray}
\label{eq: def of sh}
 s_h=-\sum_{j=0}^\infty T_h^j\op_\psi^{*-1}(qa). 
 \end{eqnarray}
Consequently
\begin{eqnarray}
\label{eq: def of rh}
 r_h=-\overline{\p}_\psi^{-1}s_h=-  \overline{\p}_\psi^{-1}\sum_{j=0}^\infty T_h^j\op_\psi^{*-1}(qa).
\end{eqnarray}
By the estimates in the proof of \cite[Lemma 3.2]{guillarmou2011identification} there is $\eps>0$ such that 
\begin{equation}\label{eq:rh_L2_estim}
 \norm{s_h}_{L^2}+\norm{r_h}_{L^2}=O(h^{1/2+\eps}),
\end{equation}
and also for $r>2$
\begin{equation}\label{eq:rh_Lr_estim}
 \norm{s_h}_{L^r}+\norm{r_h}_{L^r}=O(h^{1/r}).
\end{equation}
{\color{black} Note that we can improve on \eqref{eq:rh_Lr_estim} by {\color{black}H\"older's inequality}: for any $2<r<\infty$, choose $r'>r$ and $0<\theta<1$ such that $1/r=\theta/2+(1-\theta)/r'$. Then
\begin{equation}
||s_h||_{L^r}\leq ||s_h||_{L^2}^\theta ||s_h||_{L^{r'}}^{1-\theta}=O(h^{1/r+\theta\epsilon}),
\end{equation}
with an identical estimate holding for $r_h$. It follows that 
\begin{equation}\label{eq:rh_Lr_interpolation_estimate}
 \norm{s_h}_{L^r}+\norm{r_h}_{L^r}=O(h^{1/r+\eps_r}).
\end{equation}
for any $r\geq 2$.
}{\color{black}
Making use of the fact that $\partial\bar\partial$ is a Calder\'on-Zygmund operator and apply the standard Calder\'on-Zygmund estimates, we also have that
\begin{equation}\label{eq:CZ_rh_norms}
||r_h||_{L^r},  ||\partial r_h||_{L^r}, ||\bar\partial r_h||_{L^r}=O(h^{1/r+\epsilon_r}),
\end{equation}
for any $r\in[2,\infty)$ and $\epsilon_r >0$ depending on $r$.
}
This completes the construction of the CGO $v$.

Next we construct another CGO. By changing the sign of $\Phi$ and taking complex conjugate we also have solutions of the form
\[
 e^{-\overline \Phi/h}(\overline a+\tilde r_h),
\]
where 
\[
 \tilde r_h=-\p_\psi^{-1}\sum_{j=0}^\infty \tilde T_h^j(\p_{\psi}^{*-1} (q a))
\]
with $\tilde T_h =  \p_{\psi}^{*-1}q\p_\psi^{-1}$

(Note that $\tilde T_h$ is not the same as $\overline T_h$ since we also changed the sign of $\Phi$.)

\subsection{Choices of solutions}\label{subsection:choices_of_solutions}
Let $z_0\in \Sigma$ be fixed. For later reference, we choose three solutions to be used in the context of the second linearization as
{\color{black}\begin{align}\label{eq:sols_for_second}
 \begin{split}
 v^1&=e^{\Phi_1/h}(a+r_h), \\
 v^2&=e^{\Phi_2/h}(a+r_h), \\
 v^3&=e^{\Phi_3/h}(\overline a+\tilde r_h),
 \end{split}
\end{align}}
where in local holomorphic coordinates centered at $z_0$ 
{\color{black}\begin{align*}
\begin{split}
\Phi_1&=\frac{1}{2}z^2+O(z^3) \text{ is holomorphic }, \quad \Phi_2=\Phi_1, \\
\Phi_3&=-\z^2+O(\z^3) \text{ is antiholomorphic}.
\end{split}
\end{align*}}
We also choose four solutions to be used in the third linearization as
    \begin{align}\label{eq:sols_for_4th}
    \begin{split}
    v_1&=v_3=e^{\Phi/h}(a+r_h), \\
    v_2&=v_4=e^{-\overline \Phi/h}(\overline a+\tilde r_h), 
    \end{split}
    \end{align}
where $\Phi$ is a holomorphic Morse function with critical point at $z_0$.  In both of the above cases, we take that the holomorphic function $a$ such that it has the expansion 
    \[
    a(z)=1+O(z^N)
    \]
    in local holomorphic coordinates centered at $z_0$, $N$ large, and that $a$ vanishes to high order at any other critical point.
This is possible by \cite[Lemma 2.2.4]{guillarmou2011calderon}.

\subsection{Calculus for the CGOs}
In this section we develop a calculus involving the remainder terms $r_h$ and $\tilde r_h$. When we apply the CGOs in our inverse problem, the contributions from the remainder terms are quite diverse. Dealing with this systematically requires us to develop an abstract calculus for these CGOs which could potentially have broader applications beyond our specific inverse problem.

\begin{definition}
 Let $z_0\in \Sigma$ and 
 $$f\in C^\infty(M\setminus \{z_0\}) \cup  C^\infty(M\setminus \{z_0\}; T^*_{0,1}M)\cup  C^\infty(M\setminus \{z_0\}; T^*_{1,0}M)$$
 with $\supp (f) \subset M$. Let $z$ be holomorphic coordinates on a neighborhood of $z_0$ with $z_0$ corresponding to $0$. We say that $f$ has \emph{degree} $\deg(f) = N \in \mathbb Z$ at $z_0$ if, when expressed in local coordinates, $f$ is a finite sum of terms of the form
 \[ f = 
  Cz^k\overline z^l dz+ O(\abs{z}^{k+l+1})
 \]
with $l + k = N$, and $C$ a constant.
\end{definition}
 
Let $\psi$ be the imaginary part of any Morse holomorphic function. For sections $f$ of $T^*_{0,1}M$, denote by $\frac{f}{\bar\partial \psi}$ to be the scalar valued function satisfying $\bar\partial \psi \left (\frac{f}{\bar\partial \psi}\right) = f$. The quantity $\frac{f}{\partial \psi}$ can be defined analogously for sections of $T^*_{1,0}M$.

We have the following lemma, which can be proved by following Prop 3.4 of \cite{imanuvilov2012partial}
\begin{lemma}
\label{lem: iuy modified}
If $f\in  C^\infty(M\setminus \{z_0\}; T^*_{0,1}M)$  $\deg(f) = 1$ at $z_0$ and $\psi$ has a Morse critical point at $z_0$ then
$$\bar\partial_\psi^{-1} f = he^{-2i\psi/h} \frac{i}{2} \frac{1}{\bar\partial\psi} f + o_{L^p}(h)$$
for all $p\in[1,\infty)$. 
\end{lemma}

\begin{definition}
Let $f=f_h$ be an (a priori) $h$-dependent function defined on $M$. We now define a notion of \emph{$h$-degree of $f$}, which is set valued. We say that $f$ has \emph{$h$-degree} $(-\infty, r_0)\subset \R$ and write $\deg_h(f) = (-\infty, r_0)$  if for any $\chi \in C_c^\infty(M)$
 \[
  \int e^{4i\psi/h} f=O(h^r)
 \]
 for any $r\in (-\infty, r_0)$. We define $\deg_h(f) = (-\infty, r_0]$ if the above estimate holds for
any $r\in (-\infty, r_0]$. We define similarly if $f$ is $T_{0,1}^*M$ or $T_{1,0}^*M$ valued.

\end{definition}

\begin{remark}
If $M$ is an open subset of $\mathbb C$, then we may identify $T^*_{0,1}M$ and $T^*_{1,0}M$ valued functions, sections, with $\C$-valued functions.  All the estimates in this section would hold for functions instead of sections in this case. A reader less familiar with calculus on Riemannian surfaces can think about the computations in local coordinate charts and the sections as functions.
%
%
\end{remark}

From now on we assume that $z_0$ is fixed as the critical point of $\psi$. Furthermore, by using smooth cutoffs, we may assume that all sections are supported only in a neighbourhood of $z_0$. We need the following generalization of \cite[Proposition 2.7]{imanuvilov2010calderon} and \cite[Proposition 3.4]{imanuvilov2012partial}.  
\begin{lemma}\label{lem:expansion}
 Let $f\in C^\infty(M\setminus \{z_0\}; T^*_{0,1}M)$ with $\supp(f) \subset M$. Define $F^j$, $j=1,2,\ldots$ iteratively by
 \[
  F^1=\frac{i}{2}\frac{f}{\overline\p \psi}\in C^\infty(M\setminus \{z_0\})
 \]
 and
 \[
  F^{j+1}=\frac{i}{2}\frac{1}{\overline\p \psi}\overline\p F^j\in C^\infty(M\setminus \{z_0\}).
 \]
 
 {\color{black}

 \noindent a) If $\deg(f) = K$, $\chi \in C^\infty_c(M; T^*_{1,0}M)$, and $\rho(\ccdot,\ccdot)\in C^\infty(M\times M)$, then
$$\deg_h(f) =
    \begin{cases}
      (-\infty,   \lfloor K/2\rfloor+1 ], & \text{if } K \text{odd}\\
      (-\infty,  K/2 + 1) , & \text{if } K \text{even}
    \end{cases} $$
 and 
$$\left\|\int_M \rho(z,\ccdot) f(z)\wedge \chi(z) \right\|_{L^p(M)} \leq
    \begin{cases}
        h^{  \lfloor K/2\rfloor+1 }, & \text{if } K \text{odd}\\
        h^{K/2 + 1 - \epsilon} , & \text{if } K \text{even}
    \end{cases}$$
for all $\epsilon >0$ and $p\in [1,\infty)$.\\
}
 
 \noindent b) If  $\deg(f)=2K+1$, $K\in \mathbb N\cup\{0\}$ then 
 \begin{align}
  \op_{\psi}^{-1}(f)=e^{-2i\psi/h}\sum_{j=1}^{K+1}h^j F^j+ h^{K+1}\op^{-1}_{\psi}(\op F^{K+1}),
 \end{align}
where
 \[
   \op_{\psi}^{-1}(\op F^{K+1})=o_{L^p}(1)
 \]
 for all $p\in [1,\infty)$.\\
 \noindent c) If $\deg(f)=2K$, $K\in \mathbb N$, then 
 \begin{align}
  \op_{\psi}^{-1}(f)=e^{-2i\psi/h}\sum_{j=1}^{K} h^j F^j+ h^{K}\op^{-1}_{\psi}(\op F^{K}),
 \end{align}
where
 \[
   \op_{\psi}^{-1}(\op F^{K})=o_{L^2}(h^{1/2}).
 \]

We have analogous result for $\p_\psi^{-1}$. 

\end{lemma}

\begin{proof}
Observe first that since $\psi$ is Morse, we have $\op \psi=c\ol z+O(\abs{z}^2)$, with $c\neq 0$. Thus 
\[
 \deg(F^1)=\deg(f/\op \psi)=\deg(f)-1,
\]
and we see by a simple induction argument that
\[
 \deg(F^j)=\deg(f)-2j+1.
\]

{\color{black}
\textbf{a)} 
Without loss of generality $\chi$ is supported in a single coordinate patch. If $K = 0$ we observe that in local coordinates,
\begin{eqnarray}
\label{split K = 0 case}
\int e^{i\psi/h} \chi\wedge f = \int e^{i\psi/h} \rho_h \chi f + \int e^{i\psi/h} (1-\rho_h) \chi f
\end{eqnarray}
where $\rho\in C^\infty_c(B_1(z_0))$ with $\rho=1$ on $B_{1/2}(z_0)$and $\rho_h(z) := \rho(z/h^{1/2})$. Trivially estimating the first term of \eqref{split K = 0 case} gives
\begin{eqnarray}
\label{split K = 0 case'}
\int e^{i\psi/h} \chi f =  \int e^{i\psi/h} (1-\rho_h) \chi f + O(h).
\end{eqnarray}
We write the first term as 
$$\int e^{i\psi/h}(1-\rho_h) \chi f = h\int e^{i\psi/h} \bar\partial^* \frac{ (1-\rho_h)\chi f}{\bar\partial \psi} = O(h|\log h|)$$
via a direct calculation in polar coordinates. So \eqref{split K = 0 case'} becomes
$$\int e^{i\psi/h} \chi f = O(h|\log h|)$$
which proves the case for $K = 0$.

If $K = 1$ then for any $\chi\in C^\infty_c(M)$ we have  
$$\int e^{i\psi/h} \chi f = h\int e^{i\psi/h} \bar\partial^* \frac{\chi f}{\bar\partial \psi}$$
with an absolutely integrable integrand which proves the case when $K = 1$.

For the inductive step, suppose the statement holds for all $K\leq j$ for some $j\in \mathbb N$ and assume $\deg(f) = j+1$. Then 
$$\int e^{i\psi/h} \chi f = h\int e^{i\psi/h} \bar\partial^* \frac{\chi f}{\bar\partial \psi}.$$
Observe that $\deg (\bar\partial^* \frac{\chi f}{\bar\partial \psi}) = j+1-2 = j-1$. The statement about $\deg_h(f)$ now follows by applying the induction hypothesis. The $L^p$ estimate now follows from the statement about $\deg_h(f)$ and the dominated convergence theorem.

}

\noindent \textbf{b)} We prove the claim by induction in $K$. By the proof of \cite[Proposition 2.7]{imanuvilov2010calderon} (see also \cite[Proposition 3.4]{imanuvilov2012partial})
\[
 \op_{\psi}^{-1}f=he^{-2i\psi/h} \frac{i}{2}\frac{f}{\op \psi} +h  \op^{-1}_{\psi}\left(\op \left(\frac{i}{2}\frac{f}{\op \psi}\right)\right).
\]
where $\op^{-1}_{\psi}(\op \frac{i}{2}\frac{f}{\op \psi})=o_{L^p}(1)$. Thus the claim holds for $K=0$.

Assume then that the claim holds for all $K_0\leq K-1$. That is
\begin{equation}\label{eq:induction_expansion}
 \op_{\psi}^{-1}f=e^{-2i\psi/h}\sum_{j=1}^{K_0+1}h^j F^j+ h^{K_0+1}\op^{-1}_{\psi}(\op F^{K_0+1}).
\end{equation}
Since 
\[
\deg(F^{K_0+1})=\deg(f)-2(K_0+1)+1
=2(K-K_0)\geq 2, 
\]
we have
 \[
  \op F^{K_0+1}=O(\abs{z}).
 \]
  Thus we may apply Lemma \ref{lem: iuy modified}
\begin{equation}\label{eq:induction_step}
  \op^{-1}_{\psi}(\op F^{K_0+1})=he^{-2i\psi/h} \frac{i}{2}\frac{1}{\op \psi} \op F^{K_0+1}+h\op^{-1}_{\psi}\left(\op \frac{i}{2}\frac{1}{\op \psi} \op F^{K_0+1}\right),
 \end{equation}
 where $\op^{-1}_{\psi}\left(\op \frac{i}{2}\frac{1}{\op \psi} \op F^{K_0+1}\right)=o_{L^p}(1)$. 
Since 
\[
 \frac{i}{2}\frac{1}{\op \psi} \op F^{K_0+1}=F^{K_0+2},
\]
combining with \eqref{eq:induction_expansion} and \eqref{eq:induction_step} we have
\[
 \op_{\psi}^{-1}(f)=e^{-2i\psi/h}\sum_{j=1}^{K_0+2}h^j F^j+ h^{K_0+2}\op^{-1}_{\psi}(\op F^{K_0+2})
 \]
 with
\[
 \op^{-1}_{\psi}(\op F^{K_0+2})=o_{L^p}(1).
\]
This completes the induction step.

\noindent\textbf{c)} Let us then assume that $\deg(f)=2K$. If $K=0$ we apply Lemma  \ref{lem:log_lemma} to get
\[
 \op_\psi^{-1}(f)=o_{L^2}(h^{1/2}),
\]
which proves the claim for $K=0$. 

Assume then that $K\geq 1$. Then, by part a) applied for $2K-1 = 2(K-1) + 1$ we have
\[
 \op_{\psi}^{-1}(f)=e^{-2i\psi/h}\sum_{j=1}^{K}h^j F^j+ h^{K}\op^{-1}_{\psi}(\op F^{K}).
\]
Since 
\[
 \deg(\op F^{K})=\deg(F^{K})-1=2K-2K+1-1=0,
\]
we have by Lemma \ref{lem:log_lemma}
\[
 \op^{-1}_{\psi}(\op F^{K})=o_{L^2}(h^{1/2}).
\]

\end{proof}
 
 \begin{definition}
 We say that $f\in  C^\infty(M\setminus \{z_0\}; T^*_{0,1}M)$ is $h$-homogeneous of degree $r$ if $f$ is a finite sum of the form
 \[
  f=\sum_m h^{l_m} f_m,
 \]
 where for each $m$ and $f_m$ in the summand satisfies $r \in l_m + \deg_h(f_m)$ 
and the functions $f_m$ are independent of $h$. 
 
 If $\deg (f_m)$ is odd (respectively even) for all $m$, we say that $f$ is oddly (resp. evenly) $h$-homogeneous, or that $f$ is odd (resp. even).  We define similarly for functions of $M$ and sections of $T_{1,0}^*M$. 
 \end{definition}

 \begin{lemma}\label{lem:expansion2}
 Let $f\in C^\infty(M\setminus \{z_0\})$ with $\supp(f)\subset M$ be independent of $h$ and $K\in \N\cup \{0\}$.
 
 \noindent a) If $\deg(f)=2K+1$, then 
 \[
  \op_\psi^{-1}(f)=e^{-2i\psi/h}\tilde f+ h^{K+1}o_{L^p}(1),
 \]
where $\tilde f$ is evenly $h$-homogeneous and $ \deg_h(\tilde f)=(-\infty, K+2)$.\\
\noindent b) If $\deg(f)=2K$, then
\[
  \op_\psi^{-1}(f)=e^{-2i\psi/h}\tilde f+ h^{K}o_{L^2}(h^{1/2}), 
 \]
 where $\tilde f$ is oddly $h$-homogeneous and $\deg_h(\tilde f) =(-\infty, K+1]$.

We have analogous result for the operator $\p_\psi^{-1}$. 
\end{lemma}
\begin{proof}
\textbf{a)} Assume first that $\deg(f)=2K+1$. 
 By Lemma \ref{lem:expansion} part b), we have
 \begin{align}
  \op_{\psi}^{-1}(f)=e^{-2i\psi/h}\sum_{j=1}^{K+1}h^j F^j+ h^{K+1}o_{L^p}(1),
 \end{align}
 where 
 \[
  \deg(F^j)=2K+2-2j
 \]
 is even. By {\color{black} Lemma \ref{lem:expansion} part a)}
\[
 \deg_h(h^jF^j)=(-\infty, K+2 )
\]
independently of $j$.
Thus
\[
 \op_{\psi}^{-1}(f)=e^{-2i\psi/h}\tilde f + h^{K+1}o_{L^p}(1),
\]
where $\deg_h(\tilde f) =(-\infty, K+2)$ as claimed.

\textbf{b)} Assume then that $\deg(f)=2K$. By Lemma \ref{lem:expansion} part c), we have
\[
 \op_{\psi}^{-1}(f)=e^{-2i\psi/h}\sum_{j=1}^{K}h^j F^j+ h^K o_{L^2}(h^{1/2}),
\]
where $\deg(F^j) = 2K +1 -2j$ is odd. We now have by Lemma \ref{lem:expansion} part a),
\begin{align*}
 \deg_h(h^jF^j)&=(-\infty, K+1]
\end{align*}
 This concludes the proof.
\end{proof}

\begin{lemma}\label{lem:expansion3}
 Let $f=f_h\in  C^\infty(M\setminus \{z_0\}; T^*_{0,1}M)$ with $\supp(f) \subset M$ and $r\in \deg_h(f)$.
 
 \noindent a) If $f$ is oddly $h$-homogeneous,  then 
 \[
  \op_\psi^{-1}(f)=e^{-2i\psi/h}\tilde f+ o_{L^p}(h^r),
 \]
where $\tilde f$ is evenly $h$-homogeneous and 
\[
 \deg_h(\tilde f)=\deg_h^\circ(f) +1
\]
and $\deg_h^\circ(f)$ denotes the interior of the interval $\deg_h(f)$.\\
\noindent b) If $f$ is evenly $h$-homogeneous with $r\in \deg_h(f)$, then 
 \[
  \op_\psi^{-1}(f)=e^{-2i\psi/h}\tilde f+ h^{r-1}o_{L^2}(h^{1/2}).
 \]
where $\tilde f$ is oddly $h$-homogeneous and 
\[
 \deg_h(\tilde f)=\deg_h(f). 
\]
\end{lemma}
\begin{proof}
\textbf{a)} By assumption there are integers $l_1,l_2,\ldots$ such that
 \[
 f=\sum_{m} h^{l_m} f_m,               
\]
where the functions $f_m$ are independent of $h$ and $\deg(f_m)$ is odd. For all $m$ we have $ \sup \left(\deg_h(f)\right)= \sup \left(\deg_h( f_m) \right) + l_m$ is independent of $m$ by assumption. Using {\color{black} Lemma \ref{lem:expansion} part a)} we see that $\deg_h(f)=(-\infty, l_m+\lfloor \deg(f_m)/2\rfloor +1]$. By Lemma \ref{lem:expansion2} part a)
\begin{align*}
 \op_\psi^{-1}(f)&=\sum_{m} \left(h^{l_m}e^{-2i\psi/h}\tilde f_m+ h^{l_m}o_{L^p}(h^{\deg_h(f_m)}\right),
\end{align*}
where $\tilde f_m$ is evenly $h$-homogeneous and 
\[
 \deg_h(\tilde f_m)=(-\infty, \lfloor \deg( f_m)/2\rfloor+2)= \deg_h^\circ(f)-l_m+1.
\]
Thus 
\[
 \deg_h(h^{l_m}\tilde f_m)=\deg_h^\circ(f)+1 
\]
independently of $m$.
Writing $\tilde f=\sum_m h^{l_m}\tilde f_m$ proves part a). 

\noindent{\textbf{b)}}
Setting $K_m = \deg(f_m)$ even, we have that by Lemma \ref{lem:expansion} part a),  
$$\deg_h(f) =(\infty, l_m +K_m/2 +1). $$
Note that this quantity is independent of $m$ by our assumption that $f$ is $h$-homogeneous. Using Lemma \ref{lem:expansion2} part b),
\begin{align*}
 \op_\psi^{-1}(f)&=\sum_{m} \left(h^{l_m}e^{-2i\psi/h}\tilde f_m+ h^{l_m}h^{K_m/2}o_{L^2}(h^{1/2})\right),
\end{align*}
where $\tilde f_m$ is oddly $h$-homogeneous and 
\[
 \deg_h(\tilde f_m)=(-\infty, K_m/2+1).
\]
Thus $\deg_h(\tilde f) = \deg_h(h^{l_m}\tilde f_m)=(-\infty, K_m/2 + l_m +1) = \deg_h(f)$. This proves part b).
\end{proof}

We have now tools to analyze terms of the form $\abs{\p \Phi}^4\abs{a}\overline a r_h$ and other ones, which only involves single $r_h$ in their expressions. Recall the formula for $r_h$:
\begin{eqnarray}
\label{eq: formula of rh}
 r_h=-  \overline{\p}_\psi^{-1}s_h,\ s_h =\sum_{j=0}^\infty T_h^j\op_\psi^{*-1}(qa).
\end{eqnarray}


The next result states that with regards to stationary phase, the corresponding oscillatory integral applied to $r_h$ decays slightly faster than expected:
\begin{proposition}\label{prop:correct_decay_single_rh}
 Let 
 $f\in C^\infty(M\setminus \{z_0\})$ with $\supp(f)\subset\subset M$, then 
 \[
  \int e^{4i\psi/h}f r_h=o(h^{\lfloor \deg(f)/2\rfloor +1}).
 \]
\end{proposition}
\begin{proof}
It is obvious that it is sufficient to prove the claim for $\deg(f)$ even. 
Thus, we assume that $\deg(f)=2K$ and $K\geq 0$. Lemma \ref{lem:expansion2} part b) and integration by parts yield
 \begin{eqnarray}
\label{eq: first ibp}
 \int e^{4i\psi/h}f r_h=-\int (\op_\psi^{-1}e^{4i\psi/h}f) s_h=-\int (\p_\psi^{*-1}f)\wedge s_h \\\nonumber
 =-\int e^{2i\psi/h}\tilde f_0\wedge s_h + h^K\int o_{L^2}(h^{1/2})\wedge s_h,
\end{eqnarray}
where $\tilde f_0 =\sum_m h^{l_m}f_m,$ is oddly $h$-homogeneous,
and 
\begin{eqnarray}\label{degh of tilde f0}
\deg_h(\tilde f_0)=\deg_h(h^{l_m}f_m)=\deg_h(f)=(-\infty,K+1).
\end{eqnarray} (We remark that if $K=0$, then $\tilde f_0\equiv 0$.)

Since $s_h=O_{L^2}(h^{1/2+\eps})$ by \eqref{eq:rh_L2_estim}, we have 
\[
 h^K\int o_{L^2}(h^{1/2})\wedge s_h= o(h^{K+1})
\]
so that \eqref{eq: first ibp} becomes
\begin{eqnarray}
\label{eq: first ibp'}
\int e^{4i\psi/h}f r_h =-\int e^{2i\psi/h}\tilde f_0 \wedge s_h  + o(h^{K+1}).
\end{eqnarray}

For the first term in \eqref{eq: first ibp'}, since $\norm{T_h}_{L^2\to L^2}=O(h^{1/2-\eps})$ and $\op_\psi^{*-1}(qa)=O_{L^2}(h^{1/2+\eps})$, we may take the sum in the definition of $s_h$ in \eqref{eq: formula of rh} to be finite and absorb the remaining terms of the infinite sum into the term $o(h^{K+1})$ of \eqref{eq: first ibp'}. Thus, for some $J\in \mathbb N$ sufficiently large (depending on $K$),
\begin{eqnarray}
\label{eq: sum plus remainder}
 \int e^{4i\psi/h}f r_h=\sum_{j=0}^J \int e^{2i\psi/h}\tilde f_0\wedge T_h^j\op_\psi^{*-1}(qa)+ o(h^{K+1}).
\end{eqnarray}
We begin by showing that the $j =0$ term is of order $o(h^{K+1})$. To this end, we take adjoint of $\op_\psi^{*-1}$ to get
\begin{eqnarray}
\label{f0 term decay}
\int_M e^{2i\psi/h}\tilde f_0\wedge \op_\psi^{*-1}(qa) = \int_M qae^{-2i\psi/h}\bar\partial^{-1}_\psi \tilde f_0.
\end{eqnarray}
By applying Lemma \ref{lem:expansion3} part a) then \eqref{degh of tilde f0},
$$\deg_h(\bar\partial^{-1}_\psi \tilde f_0) =(-\infty, K+2).$$
 Since $q\in C^\infty_c(M)$ we have that \eqref{f0 term decay} becomes
$$\int_M e^{2i\psi/h}\tilde f_0\op_\psi^{*-1}(qa) = o(h^{K+1})$$
so \eqref{eq: sum plus remainder} becomes
\begin{eqnarray}\label{eq2: sum plus remainder} 
\int e^{4i\psi/h}f r_h=\sum_{j=1}^J \int e^{2i\psi/h}\tilde f_0\wedge T_h^j\op_\psi^{*-1}(qa)+ o(h^{K+1 }).
\end{eqnarray}
To deal with the finite sum involving $T_h^j$, we demonstrate the case when $j =1$ as the remaining terms can be treated analogously. We take the adjoint of $T_h$ defined by \eqref{adjoint of T} to get
\begin{eqnarray}
\label{int of tilde f0 with Th}
\int e^{2i\psi/h}\tilde f_0 \wedge T_h\op_\psi^{*-1}(qa) = \int e^{-2i\psi/h}\left(\p^{*-1}e^{2i\psi/h} \p^{-1} (e^{2i\psi/h} \tilde f_0) \right)\wedge\op_\psi^{*-1}(qa).
\end{eqnarray}
Now use Lemma \ref{lem:expansion3} parts a) and b) to show that 
$$\p^{*-1}\Big(e^{2i\psi/h} \p^{-1} (e^{2i\psi/h} \tilde f_0)\Big) = e^{2i\psi/h}\tilde f +h^{K+1-\epsilon} O_{L^2}(|h\log h|^{1/2})+ o_{L^p}(h^{K+1-\epsilon})$$
for all $\epsilon >0$ with $\tilde f$ oddly $h$-homogeneous and $ \deg_h(\tilde f_0)\subset \deg_h(\tilde f)$.
 Insert this representation for $\p^{*-1}e^{2i\psi/h} \p^{-1} (e^{2i\psi/h} \tilde f_0)$ into \eqref{int of tilde f0 with Th} and use Lemma \ref{lem:log_lemma} for estimating $\op_\psi^{*-1} (qa)$, we get that
\begin{eqnarray}
\int_M e^{2i\psi/h}\tilde f_0 \wedge T_h\op_\psi^{*-1}(qa) = \int_M e^{2i\psi/h} \tilde f \wedge \op_\psi^{*-1}(qa) + o(h^{K+1})
\end{eqnarray}
Observe now that we have arrived at the same place as \eqref{f0 term decay} with $\tilde f_0$ replaced by $\tilde f$. Repeat the same argument following \eqref{f0 term decay} completes the proof.
\end{proof}

\begin{corollary}\label{cor:correct_decay_partial_rh}
 Let $f\in C_c^\infty(\Sigma)$ be supported around a coordinate chart and $\deg(f)\geq l\geq 0$, then we have the following estimate for the coordinate expression of the integral
 \[
  \int e^{4i\psi/h}f \p^l r_h=o(h^{\lfloor (\deg(f)-l)/2\rfloor +1})=o(h^{\lfloor (\deg(\p^lf)/2\rfloor +1}).
 \]
 We have analogous result for the case where $\p^l$ above is replaced by $\op^l$, or a mixture of powers of $\p$ and $\op$.
\end{corollary}
\begin{proof}
The case $l = 0$ is simply Proposition \ref{prop:correct_decay_single_rh}.  We prove the case of general $l\in \N$ by induction. By integrating by parts, we obtain
 \[
\int e^{4i\psi/h}f \p^l r_h=-   \int e^{4i\psi/h}(4ih^{-1}(\p \psi)f+  \p f) \p^{l-1}r_h,
 \]
 where $\deg((\p \psi)f)\geq \deg(\p \psi)+\deg(f)=\deg(f)+1$ and $\deg(\p f)\geq \deg(f)-1\geq 0$. By the inductive assumption we have
 \begin{align*}
  &\int e^{4i\psi/h}(4ih^{-1}(\p \psi)f+  \p f)  \p^{l-1}r_h= h^{-1} o(h^{\lfloor\deg (f) + 2 - l \rfloor/2+1}) + o(h^{\lfloor  \deg(f)  - l \rfloor/2 +1})\\&\qquad = o(h^{\lfloor  \deg(f)  - l \rfloor/2 +1}).
 \end{align*}
This completes the inductive step.
\end{proof}

Proposition \ref{prop:correct_decay_single_rh} and Corollary \ref{cor:correct_decay_partial_rh} are sufficient to deal with terms which contain one $r_h$ or $\tilde r_h$, or their derivative. We will also be needing to be able to handle terms that involve products of $r_h$ with itself or with $\tilde r_h$, and their derivatives. To this aim, we prove the following:
{\color{black}
\begin{lemma}\label{lem:polynomial_rh_estims}
 
  Let $L(r_h,\tilde r_h)$ be a polynomial in $r_h$ and $\tilde r_h$. 
  \begin{enumerate}
   \item[(1)] If $\deg(f)=4$, $f\in C^\infty(M\setminus\{z_0\})$ with $\supp(f) \subset\subset M$ ,and order of $L(r_h,\tilde r_h)$ is $\geq 1$, then \[ \int e^{4i\psi/h}f L(r_h,\tilde r_h)=o(h^3).
  \]
  \item[(2)] If $\deg(f)=3$, $f\in C^\infty(M\setminus \{z_0\}; T^*_{0,1}M)\cup C^\infty(M\setminus \{z_0\}; T^*_{1,0}M)$ with $\supp(f) \subset\subset M$, then  
  \[
  \int e^{4i\psi/h}\star\left(f \wedge\p r_h\right)L(r_h,\tilde r_h)=o(h^2), \quad  \int e^{4i\psi/h}\star\left(f \wedge\op \tilde r_h\right) L(r_h,\tilde r_h)=o(h^2).
  \]
 \end{enumerate}
\end{lemma}
\begin{proof}

{\color{black}
We will only show (1) as the proof of (2) is essentially the same. It suffices to prove the result for homogeneous monomials in $r_h$ and $\tilde r_h$. 
Note that in the case $L(r_h,\tilde r_h)$ is a monomial of order $1$, the result is a consequence of Proposition \ref{prop:correct_decay_single_rh}. 

In general, if $L(r_h,\tilde r_h)$ is a monomial of order $m\geq 2$, we have
\begin{multline*}
\int e^{4i\psi/h}f L(r_h,\tilde r_h)
=\left(\frac{h}{4i}\right)^2\int e^{4i\psi/h}\bar\partial^*\left[\frac{1}{\partial\psi}\partial^*\left(\frac{f}{\bar\partial\psi}L(r_h,\tilde r_h)\right)\right]\\
=\left(\frac{h}{4i}\right)^2\int e^{4i\psi/h}\left[ \bar\partial^*\left(\frac{1}{\partial\psi}\partial^*\left(\frac{f}{\bar\partial\psi}\right)\right)L(r_h,\tilde r_h)
+\partial^*\left(\frac{f}{\bar\partial\psi}\right)\star\left(\frac{1}{\partial\psi}\wedge\partial L(r_h,\tilde r_h)\right)\right.\\\left.
+\star\left(\partial\left(\frac{f}{|\partial\psi|^2}\right)\wedge\bar\partial L(r_h,\tilde r_h)\right)
+\frac{f}{|\partial\psi|^2}\bar\partial^*\bar\partial L(r_h,\tilde r_h)\right].
\end{multline*}
By \eqref{eq:CZ_rh_norms},
{\color{black}$
r_h,\tilde r_h, \partial r_h, \bar\partial r_h, \partial \tilde r_h, \bar\partial \tilde r_h=o_{L^c}(h^{1/c})$, for all $c\geq 2$.}
Using H\"older's inequality we conclude that
\begin{equation}
\label{main term in L}
\int e^{4i\psi/h}f L(r_h,\tilde r_h)
=\left(\frac{h}{4i}\right)^2\int e^{4i\psi/h}\frac{f}{|\partial\psi|^2}\bar\partial^*\bar\partial L(r_h,\tilde r_h)+o(h^3).
\end{equation}
Observe that
$$\bar\partial^*\bar\partial L(r_h,\tilde r_h) =  \tilde L( r_h,\tilde r_h, \partial r_h, \bar\partial r_h, \partial \tilde r_h, \bar\partial \tilde r_h) + \bar\partial^*\bar\partial r_h M(r_h, \tilde r_h) + \bar\partial^*\bar\partial \tilde r_h \widetilde M(r_h, \tilde r_h)$$
where $\tilde L$ is a sum of monomials all of which are of order $m$, while $M$ and $\widetilde M$ are sums of monomials of order $m-1$. Inserting this expression for $\bar\partial^*\bar\partial L$ into \eqref{main term in L} and applying H\"older's inequality as before, we get that 
\begin{eqnarray}
\label{main term in L'}\nonumber
\int e^{4i\psi/h}f L(r_h,\tilde r_h)
&=&\left(\frac{h}{4i}\right)^2\int e^{4i\psi/h}\frac{f}{|\partial\psi|^2}\bar\partial^*\bar\partial r_h M(r_h, \tilde r_h)\\&+&\left(\frac{h}{4i}\right)^2\int e^{4i\psi/h}\frac{f}{|\partial\psi|^2}\bar\partial^*\bar\partial\tilde r_h \widetilde M(r_h, \tilde r_h)+o(h^3).
\end{eqnarray}
Note that {\color{black}$\deg\left(\frac{f}{|\partial\psi|^2}\right) = 2$} and $\bar\partial^*\bar\partial r_h$ is given by
\begin{eqnarray*}
\bar\partial^*\bar\partial r_h
=
qa- e^{-2i\psi/h}\star \left(\frac{2i\partial\psi}{h}\wedge s_h\right) +e^{-2i\psi/h}qr_h
=qa-e^{-2i\psi/h}\star\left(\frac{2i\partial\psi}{h} \wedge s_h\right)+ o_{L^c}(h^{1/c}),
\end{eqnarray*}
for $c\geq 2$. A similar expression holding for $\bar\partial^*\bar\partial\tilde r_h = \partial^*\partial \tilde r_h$. Inserting these into \eqref{main term in L'} and using \eqref{eq:CZ_rh_norms} when $m >2$, or Proposition \ref{prop:correct_decay_single_rh} when $m = 2$, we get
\begin{eqnarray}\label{r-M-estimate}\nonumber
\int e^{4i\psi/h}f L(r_h,\tilde r_h)
&=&Ch\int e^{2i\psi/h}\frac{f}{|\partial\psi|^2}\star(\partial\psi\wedge s_h) M(r_h,\tilde r_h)+o(h)\\\nonumber
&=& Ch^2\int\bar\partial^*\left(  e^{2i\psi/h}\frac{f}{|\partial\psi|^2}s_h M(r_h,\tilde r_h) \right)\\\nonumber
&+&Ch^2\int e^{2i\psi/h}\left[\star  \left(\partial\left(\frac{f}{|\partial\psi|^2}\right)\wedge s_h\right) M(r_h,\tilde r_h)\right]\\&+& Ch^2\int e^{2i\psi/h}\left[\frac{f}{|\partial\psi|^2}(\bar\partial^* s_h) M(r_h,\tilde r_h)\right]\\\nonumber
&+& Ch^2\int e^{2i\psi/h}\left[\frac{f}{|\partial\psi|^2}\star(s_h\wedge \partial M(r_h,\tilde r_h))\right]+o(h^3).
\end{eqnarray}
The first term vanishes after integration by parts. The second term and the last term can be estimated using H\"older inequality and \eqref{eq:CZ_rh_norms} to show that they are $o(h^3)$. For the third term we can compute $\bar\partial^* s_h$ using \eqref{eq: def of sh} then apply H\"older inequality to show that it is also $o(h^3)$.

}

\end{proof}
}

\section{Asymptotic analysis of the integral identities}\label{Section_4}
In this section we analyse the integral identities corresponding to the second and third linearization. First let $\Phi_1=\Phi_2=\Phi$ where $\Phi$ is a holomorphic Morse function, and let $\Phi_3=-2\bar\Phi$. As described in section \ref{subsection:choices_of_solutions}, we construct solutions 
 \begin{align*}
 \begin{split}
 v^1&=e^{\Phi_1/h}(a+r_h), \\
 v^2&=e^{\Phi_2/h}(a+r_h), \\
 v^3&=e^{\Phi_3/h}(\overline a+\tilde r_h).
 \end{split}
\end{align*}
 We also denote $\hat v_j=e^{\Phi_j/h}a$, $j=1,2,3$. The first result of this section is:

\begin{proposition}\label{prop:asymptotics_for_the_2nd_lin}
 Let $(\Sigma,g)$ be a Riemannian surface with boundary. 
 Assume that $K$ is a $2$-tensor that vanishes to infinite order on the boundary. We have the expansion
 \begin{multline}
  \int_{\Sigma} v_1 K(\nabla v_2,\nabla v_3\big)dV+\int_{\Sigma}v_2 K(\nabla v_1,\nabla v_3\big)dV+\int_{\Sigma} v_3 K(\nabla v_1,\nabla v_2\big)dV= \\
  \int_{\Sigma} \hat v_1 K(\nabla \hat  v_2,\nabla \hat v_3\big)dV+\int_{\Sigma}\hat v_2 K(\nabla \hat v_1,\nabla \hat v_3\big)dV+\int_{\Sigma}\hat v_3 K(\nabla\hat v_1,\nabla \hat v_2\big)dV+o(1).
 \end{multline}
\end{proposition}

Let us then present the other result of this section. Let $\Phi$ be a holomorphic Morse function with a critical point at $P$, and $a$ a holomorphic function that is equal to $1$ to high order in a neighborhood of $P$ and vanishes to high order close to all other critical points of $\Phi$. Let
\begin{align*}
    \begin{split}
    v_1&=v_3=e^{\Phi/h}(a+r_h), \\
    v_2&=v_4=e^{-\overline \Phi/h}(\overline a+\tilde r_h)
    \end{split}
    \end{align*}
    be solutions to the linearized equation as in \eqref{eq:sols_for_4th}.
\begin{proposition}\label{prop:asymptotics_for_the_3rd_lin}
 Let $(\Sigma,g)$ be a Riemannian surface with boundary.  Assume that $Q$ vanishes to infinite order on the boundary. We have the expansion 
 \begin{multline*}
  0=\int_{\Sigma}Q\Big[g(\nabla v_1\cdot \nabla v_2)g(\nabla v_3\cdot \nabla  v_4)+g(\nabla v_1\cdot \nabla v_3)g(\nabla v_2\cdot \nabla  v_4) \\
 +g(\nabla v_2\cdot \nabla v_3)g(\nabla v_1\cdot \nabla  v_4)\Big]dV=ch^{-1}Q(P)+o(h^{-1}),
 \end{multline*}
 where $c\neq 0$ is a constant.
\end{proposition}

\subsection{Proof of Proposition \ref{prop:asymptotics_for_the_2nd_lin}}

It suffices to verify that 
\begin{eqnarray}
\label{eq: first term with K_cc}
\int_{\Sigma}v_1K(\nabla v_2, \nabla v_3)d V - \int_{\Sigma}\hat v_1K(\nabla\hat  v_2, \nabla\hat v_3)d V = o(1)
\end{eqnarray}
as the other terms in \eqref{prop:asymptotics_for_the_2nd_lin} can be analysed by the  same procedure. Note that
\begin{equation*}
\nabla v_1=\left(h^{-1}\nabla\Phi_1(a+r_h)+\nabla a+\nabla r_h\right) e^{\Phi_1/h},
\end{equation*}
with similar expressions holding for $v_2$ and $v_3$. Plugging these into the left hand side of \eqref{eq: first term with K_cc} and grouping the terms with the same power of $h$ as a factor, we have
\begin{equation*}
\int_{\Sigma}v_1K(\nabla v_2, \nabla v_3)d V - \int_{\Sigma}\hat v_1K(\nabla\hat  v_2, \nabla\hat v_3)d V=A_h+B_h+C_h,
\end{equation*}
where
\begin{equation*}
A_h=h^{-2}\int e^{4i\psi/h}\left[ 2|a|^2 r_h+a^2\tilde r_h+\bar a r_h^2+2a r_h\tilde r_h+r_h^2\tilde r_h
  \right]K(\nabla\Phi_2,\nabla\Phi_3)d V,
\end{equation*}
\begin{multline*}
B_h=h^{-1}\int e^{4i\psi/h}\left[a^2K(\nabla\Phi_2,\nabla\tilde r_h)+(2ar_h+r_h^2)K(\nabla\Phi_2,\nabla(\bar a+\tilde r_h)) \right.\\
+\left.|a|^2K(\nabla r_h,\nabla\Phi_3)+(a\tilde r_h+\bar a r_h+r_h\tilde r_h)K(\nabla(a+r_h),\nabla\Phi_3)   \right]d V
\end{multline*}
and
\begin{multline*}
C_h=\int e^{4i\psi/h}\left[ r_h K(\nabla(a+r_h),\nabla(\bar a+\tilde r_h))\right.\\ +\left.aK(\nabla r_h,\nabla\bar a)+aK(\nabla a,\nabla\tilde r_h)+aK(\nabla r_h,\nabla\tilde r_h)  \right] dV.
\end{multline*}
Since $\nabla a$  vanishes to high order at all critical points of $\psi$, notice that we can use Lemma \ref{lem:polynomial_rh_estims} to eliminate it from the expressions of $B_h$ and $C_h$ to obtain
\begin{multline*}
B_h=h^{-1}\int e^{4i\psi/h}\left[a^2K(\nabla\Phi_2,\nabla\tilde r_h)+(2ar_h+r_h^2)K(\nabla\Phi_2,\nabla\tilde r_h) \right.\\
+\left.|a|^2K(\nabla r_h,\nabla\Phi_3)+(a\tilde r_h+\bar ar_h+r_h\tilde r_h)K(\nabla r_h,\nabla\Phi_3)   \right]d V+o(1),
\end{multline*}
and
\begin{equation*}
C_h=\int e^{4i\psi/h}\left[ r_h K(\nabla r_h,\nabla \tilde r_h)+aK(\nabla r_h,\nabla\tilde r_h)  \right] dV+o(1).
\end{equation*}

Recalling (see the estimates \eqref{eq:CZ_rh_norms}) that
\begin{equation}\label{eq:recall_rh_norms_cc}
||r_h||_{L^r}, ||\tilde r_h||_{L^r}, ||\partial r_h||_{L^r}, ||\bar\partial r_h||_{L^r}, ||\partial\tilde r_h||_{L^r}, ||\bar\partial\tilde r_h||_{L^r}=O(h^{1/r+\epsilon_r})
\end{equation}
for $r\geq 2$, 
we can immediately conclude that 
\begin{equation*}
C_h=o(1),
\end{equation*}
and also that
\begin{equation*}
B_h=h^{-1}\int e^{4i\psi/h}\left[ a^2K(\nabla\Phi_2,\nabla\tilde r_h)+|a|^2K(\nabla r_h,\nabla\Phi_3)\right]dV+o(1).
\end{equation*}
Since $\deg(\nabla\Phi_2)=\deg(\nabla\Phi_3)=1$, we can apply Corollary \ref{cor:correct_decay_partial_rh} to the remaining two terms, which then gives
\begin{equation*}
B_h=o(1).
\end{equation*}
Since $\deg(K(\nabla\Phi_2,\nabla\Phi_3))=2$, the same Corollary \ref{cor:correct_decay_partial_rh} gives that 
\begin{equation*}
h^{-2}\int e^{4i\psi/h}\left(2|a|^2 r_h+a^2\tilde r_h\right)K(\nabla\Phi_2,\nabla\Phi_3)dV=o(1)
\end{equation*}
for the first two terms appearing in $A_h$. 

Let $\chi_\alpha$, $\alpha=0,\ldots, M$, be a partition of unity for $\Sigma$, with each $\chi_\alpha$ supported in an isothermal coordinate chart. In order to complete the estimate for $A_h$, note that in each of these charts there is a function $\kappa_\alpha$ so that in local coordinates we have 
\begin{equation*}
K(\nabla\Phi_2,\nabla\Phi_3)=\kappa_\alpha\partial\Phi_2\bar\partial\Phi_3.
\end{equation*}
Using local coordinates in each of the following integrals, we compute
\begin{multline*}
A_h=\sum_{\alpha=0}^Mh^{-2}\int \chi_\alpha e^{4i\psi/h}\left(\bar a r_h^2+2a r_h\tilde r_h+r_h^2\tilde r_h\right)\kappa_\alpha\partial\Phi_2\bar\partial\Phi_3 dV\\
=\frac{1}{2}h^{-1}\sum_{\alpha=0}^M\int\left(\partial e^{4i\psi/h}\right)\left(\bar a r_h^2+2a r_h\tilde r_h+r_h^2\tilde r_h\right)\chi_\alpha\kappa_\alpha\bar\partial\Phi_3 dV+o(1)\\
=-\frac{1}{2}h^{-1}\sum_{\alpha=0}^M\int e^{4i\psi/h}\bar\partial\Phi_3\partial\left[\chi_\alpha\kappa_\alpha \left(\bar a r_h^2+2a r_h\tilde r_h+r_h^2\tilde r_h\right)\right] dV+o(1)=o(1),
\end{multline*}
where the last step follows by the estimates \eqref{eq:recall_rh_norms_cc}. In the second equality we used $\p \psi=(\Phi_2-\overline \Phi_2)/2$. Recall that $K$ is assumed to vanish to infinite order on the boundary $\partial\Sigma$, so the integration by parts in the computation above works the same in interior charts as well as in those that intersect the boundary.

\subsection{Proof of Proposition \ref{prop:asymptotics_for_the_3rd_lin}}
We prove Proposition \ref{prop:asymptotics_for_the_3rd_lin}. Let us cover $\Sigma$ with finitely many isothermal coordinate charts. Let $\chi_\alpha$, $\alpha=0,\ldots, M$, be a partition of unity associated with this covering, where the labeling is chosen so that the chart containing $P$ corresponds to $\alpha=0$. We write the integral identity 
\begin{multline}\label{eq:3rd_ord_integral_identity_proof}
0=\int_{\Sigma}Q\Big[g(\nabla v_1\cdot \nabla v_2)g(\nabla v_3\cdot \nabla  v_4)+g(\nabla v_1\cdot \nabla v_3)g(\nabla v_2\cdot \nabla  v_4) \\
 +g(\nabla v_2\cdot \nabla v_3)g(\nabla v_1\cdot \nabla  v_4)\Big]dV
 \end{multline}
 in terms of $\p$ and $\op$ operators. 

In local holomorphic coordinates, we have for some conformal factor $\gamma$ that $g=\gamma\s I_{2\times 2}$ and thus
\begin{equation}
\nabla u\cdot\nabla w=g(\nabla u,\nabla w)=2\gamma^{-1}(\partial u\bar\partial w+\bar\partial u\partial w).
\end{equation}
By using this, we obtain
\begin{multline*}
4^{-1}\gamma^2(\nabla v_1\cdot \nabla v_2)(\nabla v_3\cdot\nabla v_4)\\[5pt]
= \partial v_1\bar\partial v_2\partial v_3\bar\partial v_4
+\partial v_1\bar\partial v_2\bar\partial v_3\partial v_4
+\bar\partial v_1\partial v_2\partial v_3\bar\partial v_4
+\bar\partial v_1\partial v_2\bar\partial v_3\partial v_4,
\end{multline*}
\begin{multline*}
4^{-1}\gamma^2(\nabla v_1\cdot\nabla v_3)(\nabla v_2\cdot\nabla v_4)\\[5pt]
= \partial v_1\bar\partial v_3\partial v_2\bar\partial v_4
+\partial v_1\bar\partial v_3\bar\partial v_2\partial v_4
+\bar\partial v_1\partial v_3\partial v_2\bar\partial v_4
+\bar\partial v_1\partial v_3\bar\partial v_2\partial v_4,
\end{multline*}
\begin{multline*}
4^{-1}\gamma^2(\nabla v_1\cdot\nabla v_4)(\nabla v_2\cdot\nabla v_3)\\[5pt]
= \partial v_1\bar\partial v_4\partial v_2\bar\partial v_3
+\partial v_1\bar\partial v_4\bar\partial v_2\partial v_3
+\bar\partial v_1\partial v_4\partial v_2\bar\partial v_3
+\bar\partial v_1\partial v_4\bar\partial v_2\partial v_3.
\end{multline*}
Adding the above together yields
\begin{multline*}
4^{-1}\gamma^2 \left[(\nabla v_1\cdot \nabla v_2)(\nabla v_3\cdot\nabla v_4)\right.\\[5pt]\left.+(\nabla v_1\cdot\nabla v_3)(\nabla v_2\cdot\nabla v_4)+(\nabla v_1\cdot\nabla v_4)(\nabla v_2\cdot\nabla v_3)\right]\\[5pt]
=2\partial v_3\bar\partial v_4\left(\partial v_1\bar\partial v_2+\bar\partial v_1\partial v_2\right)
+2\bar\partial v_3\partial v_4\left(\partial v_1\bar\partial v_2+\bar\partial v_1\partial v_2\right)\\[5pt]
+2\partial v_3\partial v_4\bar\partial v_1\bar\partial v_2
+2\bar\partial v_3\bar\partial v_4\partial v_1\partial v_2,
\end{multline*}
so that
\begin{multline*}
(\nabla v_1\cdot \nabla v_2)(\nabla v_3\cdot\nabla v_4)+(\nabla v_1\cdot\nabla v_3)(\nabla v_2\cdot\nabla v_4)+(\nabla v_1\cdot\nabla v_4)(\nabla v_2\cdot\nabla v_3)\\[5pt]
=4\gamma^{-2}(\partial v_1\bar\partial v_2\partial v_3\bar\partial v_4
+\partial v_1\bar\partial v_2\bar\partial v_3\partial v_4
+\bar\partial v_1\partial v_2\partial v_3\bar\partial v_4
+\bar\partial v_1\partial v_2\bar\partial v_3\partial v_4) \\
+ 8\gamma^{-2}\left( \partial v_3\partial v_4\bar\partial v_1\bar\partial v_2
+\bar\partial v_3\bar\partial v_4\partial v_1\partial v_2\right).
\end{multline*}

Let us then substitute $v_1=v_3$ and $v_2=v_4$ to the above. This gives
\begin{multline*}
 (\nabla v_1\cdot \nabla v_2)(\nabla v_3\cdot\nabla v_4)+(\nabla v_1\cdot\nabla v_3)(\nabla v_2\cdot\nabla v_4)+(\nabla v_1\cdot\nabla v_4)(\nabla v_2\cdot\nabla v_3) \\
 =4\gamma^{-2}((\p v_1\overline \p v_2)^2+2\p v_1\overline \p v_1 \p v_2 \overline \p v_2 + (\overline \p v_1\p v_2)^2)+ 16\gamma^{-2}(\p v_1 \p v_2 \overline \p v_1 \overline \p v_2) \\
 =\gamma^{-2}(4(\p v_1\overline \p v_2)^2 + 24 \p v_1 \p v_2 \overline \p v_1 \overline \p v_2 + 4 (\overline \p v_1\p v_2)^2).
\end{multline*}
Thus we have three different types of terms,
\[
 (\p v_1\overline \p v_2)^2, \quad \p v_1 \p v_2 \overline \p v_1 \overline \p v_2 \ \text{ and } \ (\overline \p v_1\p v_2)^2
\]
to consider. We simplify notation by writing 
\[
Q_\gamma:=Q\gamma^{-2}. 
\]

\subsubsection{The term  $(\p v_1\overline \p v_2)^2$.}\label{sec:leading_term}
We will see that the term
\[
 (\p v_1\overline \p v_2)^2
\]
produces the largest contribution of size $O(h^{-1})$ to the integral identity \eqref{eq:3rd_ord_integral_identity_proof}. 
To analyze this term, we start by computing
\begin{multline*}
 \p v_1\overline \p v_2=\p(e^{\Phi/h}(a+r_h))\overline\p(e^{-\overline\Phi/h}(\overline a+\tilde r_h)) \\
 =e^{2i\psi/h}\left(-h^{-2}\abs{\p\Phi}^2(a+r_h)(\overline a+\tilde r_h) + h^{-1} \p\Phi (a+r_h) \op(\overline a+\tilde r_h)\right. \\
 -h^{-1} \op\overline\Phi \p(a+r_h) (\overline a+\tilde r_h) 
 \left.+ \p(a+r_h) \op(\overline a+\tilde r_h) \right).
\end{multline*}
We write 
\begin{align*}
 H_2&=-\abs{\p\Phi}^2(a+r_h)(\overline a+\tilde r_h), \\
 H_1&=\p\Phi (a+r_h) \op(\overline a+\tilde r_h)-\op\overline\Phi \p(a+r_h) (\overline a+\tilde r_h), \\
 H_0&=\p(a+r_h) \op(\overline a+\tilde r_h) .
\end{align*}
Then we have
\begin{align}\label{eq:Hformula}
 (\p v_1\overline \p v_2)^2=e^{4i\psi/h}\left(h^{-4}H_2^2+h^{-3}H_1H_2+h^{-2}H_0H_2+h^{-2}H_1^2+h^{-1}H_0H_1+H_0^2\right).
\end{align}

Below $L(r_h,\tilde r_h)$ will denote an unspecified polynomial in $r_h$ and $\tilde r_h$, and  $L_1(r_h,\tilde r_h)$ will denote an unspecified polynomial in $r_h$ and $\tilde r_h$ of degree $\geq 1$. We may write 
\[
 H_2^2=-\abs{\p \Phi}^4\abs{a}^4+f_{22}L_1(r_h,\tilde r_h),
\]
where $\deg(f_{22})=4$. Thus by Lemma \ref{lem:polynomial_rh_estims} part a), we have 
\[
 h^{-4}\int e^{4i\psi/h}\chi_0 Q_\gamma H_2^2=-h^{-4}\int e^{4i\psi/h} \chi_0 Q\abs{\p \Phi}^4\abs{a}^4 + o(h^{-1}).
\]
Recall that $\chi_0$ is corresponds to the chart containing $P$. 
By stationary phase (see e.g. \cite[Theorem 7.7.5]{hormander2015analysis}), we have
\[
 -h^{-4}\int e^{4i\psi/h}e^{4i\psi/h}  \chi_0  Q_\gamma\abs{\p \Phi}^4\abs{a}^4=ch^{-1}Q(P)+O(1),
\]
where $c\neq 0$. We see here the leading term $ch^{-1}Q_\gamma(P)$. 

For $\alpha \neq 0$, we have
\begin{equation}\label{eq:alpha_neq_zero}
 h^{-4}\int e^{4i\psi/h}\chi_\alpha Q_\gamma H_2^2=o(h^{-1}). 
\end{equation}
Indeed, for $\alpha\neq 0$ there are two options. The first option is that  the domain of $\chi_\alpha$  contains a critical point of $\varphi$. In this case, we may use above argument with the exception that $c=0$ since $a$ equals $0$ at the critical points other than $P$. Thus we obtain \eqref{eq:alpha_neq_zero}. The second option is that the domain of $\chi_\alpha$ does not contain a critical point of $\varphi$. In this case, we may proceed as above in this section, but we may integrate parts unrestricted amount of times to obtain
\[
 -h^{-4}\int e^{4i\psi/h} \chi_0 Q\abs{\p \Phi}^4\abs{a}^4 =O(h^\infty).
\]
Thus \eqref{eq:alpha_neq_zero} follows also in this case. Note that the presence of the boundary in any of the cases considered does not affect the result since we are assuming all the derivatives of $Q$ vanish on the boundary, so that implied integrations by parts can still be carried out in the same way. 
%
In the remaining of this section, we consider only the case $\alpha=0$ as the cases $\alpha\neq 0$ give lower order contributions by the reasoning above.

Next we analyze the term $H_1H_2$. We note that 
\[
 H_1H_2=f_{12}\p r_h L(r_h,\tilde r_h)+\tilde f_{12}\op \tilde r_h L(r_h,\tilde r_h)+\hat f_{12},
\]
where $\deg(f_{12})=\deg(\tilde f_{12})=3$ and $\deg(\hat{f}_{12})=N$, with $N$ large, since $\p a=0$ to high order at all critical points (including $P$). 
Thus by Lemma \ref{lem:polynomial_rh_estims} and stationary phase that
\[
 h^{-3}\int e^{4i\psi/h} \chi_0 Q_\gamma H_1H_2
 =o(h^{-1}).
\]

We note that the term $H_0H_2$ can be written as
\[
 H_0H_2=f_{02}\p r_h L(r_h,\tilde r_h)+\tilde f_{02}\op \tilde r_h L(r_h,\tilde r_h)+\hat f_{02}+\p r_h\op \tilde r_hL(r_h,\tilde r_h).
\]
Here $\deg(f_{02})=\deg(\tilde f_{02})=\deg(\hat f_{02})=N$, where $N$ is large since $\p a$ equals $0$ to high order at the critical points. 
We also have
\[
 h^{-1}\int e^{4i\psi/h} \chi_0 Q_\gamma \p r_h\op \tilde r_hL(r_h,\tilde r_h)=o(1)
\]
by the estimates \eqref{eq:rh_L2_estim}--\eqref{eq:rh_Lr_interpolation_estimate}. Combining this with Lemma \ref{lem:polynomial_rh_estims} shows that
\[
 h^{-1}\int e^{4i\psi/h} \chi_0 Q_\gamma H_0H_2=o(1).
\]

Similarly, the term $H_1^2$ can be written as 
\[
 H_1^2=f_{11}\p r_h L(r_h,\tilde r_h)+\tilde f_{11}\op \tilde r_h  L(r_h,\tilde r_h)+\hat f_{11}L(r_h,\tilde r_h)+\check f_{11}+\p r_h\op \tilde r_hL(r_h,\tilde r_h),
\]
where $\deg(f_{02})$, $\deg(\tilde f_{02})$, $\deg(\hat f_{02})$ and $\deg(\check f_{11})$ are large. By the same arguments as above, we obtain
\[
 h^{-2}\int e^{4i\psi/h} \chi_0 Q_\gamma H_1^2=o(h^{-1}).
\]
In fact, also $H_0H_1$ and $H_0^2$ can be written similarly as $H_0H_2$. Thus the term $H_0H_1$ can written as
\[
 h^{-1}\int e^{4i\psi/h} \chi_0 Q_\gamma H_0H_1=o(1), \quad \int e^{4i\psi/h} \chi_0 Q_\gamma H_0H_1=o(h).
\]
Combining everything so far, we have obtained
\[
 \int e^{4i\psi/h} Q_\gamma(\p v_1\overline \p v_2)^2=ch^{-1}\gamma^{-2}(x_0)Q(x_0)+o(h^{-1}).
\]

\subsubsection{The term  $\p v_1 \p v_2 \overline \p v_1 \overline \p v_2$.} \label{sec:subleading_term}
Let us then consider the term $\p v_1 \p v_2 \overline \p v_1 \overline \p v_2$. With the choices of $v_1$ and $v_2$, this is
\begin{align*}
 \p v_1&=\p (e^{\Phi/h}(a+r_h))=e^{\Phi/h}\left(\frac{\p \Phi}{h}(a+r_h)+\p(a+r_h)\right),  \\
 \op v_1&=\op (e^{\Phi/h}(a+r_h))=e^{\Phi/h}\op r_h, \\
 \p v_2&=\p (e^{-\overline\Phi/h}(\overline a+\tilde r_h))= e^{-\overline\Phi/h}\p \tilde r_h, \\
 \op v_2&=\op (e^{-\overline\Phi/h}(\overline a+\tilde r_h))=e^{-\overline\Phi/h}\left(-\frac{\overline\p \Phi}{h}(\overline a+\tilde r_h)+\p(\overline a+\tilde r_h)\right).
\end{align*}
Thus we have
%
%
\begin{multline*}
 e^{-4i\psi/h}\p v_1 \p v_2 \overline \p v_1 \overline \p v_2 =\op r_h\p \tilde r_h\left[-\frac{\abs{\overline\p \Phi}^2}{h^2}(a+ r_h)(\overline a+\tilde r_h)+\frac{\p \Phi}{h}(a+r_h)\p(\overline a+\tilde r_h)\right.\\
 \left.-\frac{\overline\p \Phi}{h}(\overline a+\tilde r_h)\op (a+r_h)+\p(\overline a+\tilde r_h)\op (a+r_h)\right]. 
\end{multline*}
Note the factor $\op r_h\p \tilde r_h$ in front. Consequently, by the estimates \eqref{eq:rh_L2_estim}--\eqref{eq:rh_Lr_interpolation_estimate} we have
\[
 \int Q_\gamma \p v_1 \p v_2 \overline \p v_1 \overline \p v_2=o(h^{-1}).
\]

\subsubsection{The term  $(\op v_1 \p v_2)^2$.} \label{sec:smallest_term}
This is the easiest term. We have
\begin{align*}
 \op v_1&=\op (e^{\Phi/h}(a+r_h))=e^{\Phi/h}\op r_h \\
 \p v_2&=\p (e^{-\overline\Phi/h}(\overline a+\tilde r_h))= e^{-\overline\Phi/h}\p \tilde r_h. 
\end{align*}
Thus by the $L^4$-estimate \eqref{eq:rh_Lr_estim},  we simply have
\[
 \left|\int Q_\gamma (\op v_1 \p v_2)^2\right| \leq C \int |\op r_h \p \tilde r_h|^2 \leq C\norm{\op r_h}_{L^4}^2\norm{\p \tilde r_h}_{L^4}^2=o(h).
\] 
%
Combining the results of the calculations in Sections \ref{sec:leading_term}--\ref{sec:smallest_term} proves Proposition \ref{prop:asymptotics_for_the_3rd_lin}.

\section{Proof of Theorem \ref{thm:main}}\label{sec:proof_of_main_thm} 
We prove Theorem \ref{thm:main}.  For this let us assume that $(\Sigma_1,g_1)$ and $(\Sigma_2,g_2)$ are $2$-dimensional Riemannian manifolds with mutual boundary $\p\Sigma$, that there are $1$-parameter families of Riemannian metrics $g_a(\ccdot,s)$ on $\Sigma_\beta$, that $0$ is a solution to minimal surface equation \eqref{eq:minimal_surface_general} on both $\Sigma_\beta$ and and that $\Lambda_{g_1}=\Lambda_{g_2}$. Here and below the index $\beta=1,2$ refers to quantities on the manifold $\Sigma_\beta$. 
The proof will be achieved by considering the first three linearizations of the minimal surface equation. 
\subsection{Step 1: First linearization}

By Lemma \ref{lem:high_ord_lin}, the first linearization of the minimal surface equation is  
\begin{equation}\label{eq:first_lin_proof}
	\begin{aligned}
		\begin{cases}
			(\Delta_{g_\beta}+h_{\beta}^{(1)}/2)v_\beta^{j}=0 
			& \text{ in } \Sigma_\beta,
			\\
			v_\beta^{j}=f_j
			&\text{ on }\p \Sigma.
		\end{cases}
	\end{aligned}
\end{equation}
 The index $j=1,\ldots,4$, refers to the boundary value $f_j\in C^\infty(\p \Sigma)$, which we assume to be the same for both $\beta=1,2$.  As remarked in Section \ref{Section 2}, the DN map of the minimal surface equation is smooth in Frech\'et sense. It thus follows from $\Lambda_{g_1}=\Lambda_{g_2}$, that the DN maps of the first linearizations \eqref{eq:first_lin_proof} also agree. 
 
 Since we assume that $\Sigma_1$  and $\Sigma_2$ are diffeomorphic to a bounded domain in $\R^2$ by boundary fixing maps, we may consider that $\Sigma_1$ and $\Sigma_2$ are a domain $\Omega$ in $\R^2$ while the DN maps of \eqref{eq:first_lin_proof} still agree for $\beta=1$ and $\beta=2$. The latter follows from uniqueness of solutions to corresponding Dirichlet problems. (Note that a diffeomorphism maps $\p \Sigma$ to $\p \Omega$.)  Consequently,  by \cite[Theorem 1.1]{imanuvilov2012partial} there is a diffeomorphic conformal mapping
\[
 F:\Sigma_1\to\Sigma_2, \quad F^*g_2=cg_1,
\]
which also satisfies $F|_{\p \Sigma}=\text{Id}$, and 
\[
 h_{1}^{(1)}=c F^*h_{2}^{(1)}
\]
Especially $F$ preserves Cauchy data of functions. The conformal factor $c$ satisfies $c|_{\p \Sigma}=1$. Here $F^*$ denotes the pullback by $F$. We mention that this is the only part of the proof where we use that $\Sigma_1$ and $\Sigma_2$ are diffeomorphic to $\Omega$ by boundary fixing maps.

We use $F$ to transform our analysis to $(\Sigma_1,g_1)$ as follows. We simplify our notation by setting
\[
 v^{j}:=v^{j}_1, \quad \widetilde v^{j}:=v_2^{j}\circ F
\]
and 
\[
  w^{jk}:=w^{jk}_1, \quad \widetilde w^{jk}_{2}=w_2^{jk}\circ F \ \text{ and } \ w^{jkl}:=w^{jkl}_1, \quad \widetilde w^{jkl}_{2}=w_2^{jkl}\circ F.
\]
and 
\[
 h^{(1)}:=h_{1}^{(1)} \quad \widetilde h^{(1)}:=h_{2}^{(1)}\circ F.
\]
We use similar notation for other quantities throughout out the proof. We also note that
 \[
  d\widetilde V_2:= F^*dV_2=c^{\dim(\Sigma_1)/2}dV_1=cdV_1.
 \]
We will also denote $dV:=dV_1$ and $\Sigma=\Sigma_1$.

Since $F$ is the identity on the boundary and
\begin{multline}\label{eq:morphism_of_sols}
 (\Delta_{g_1}+h_1^{(1)}/2)\widetilde v^{j}=\Delta_{c^{-1}\s F^*g_2}\widetilde v^{j}+(h_1^{(1)}/2)\widetilde v^{j}=c\s \Delta_{F^*g_2}v_2^{j}\circ F+(c F^*h_{2}^{(1)}/2)v_2^{j}\circ F \\
 =c\s F^*((\Delta_{g_2}+h_2^{(1)}/2)v_2)=0,
\end{multline}
we see that both $v^{j}_1$ and $v_2^{j}\circ F$ satisfy the same equation $(\Delta_{g_1}+h_1^{(1)}/2)v=0$ on $\Sigma_1$ and have the same boundary value $f_j$. Here we used that the Laplace-Beltrami operator in dimension $2$ is conformally invariant. By uniqueness of solutions, we thus have
\[
 v^{j}=\widetilde v^{j}.
\]

We record the following lemma:
\begin{lemma}\label{lem:F_id_infty}
 Let $(\Sigma_1,g_1)$ and $(\Sigma_2,g_2)$ be compact Riemannian surfaces with a mutual boundary $\p \Sigma$. Let also $L_\beta=\Delta_{q_\beta}+q_\beta$, $\beta=1,2$, be such that the corresponding DN maps of $L_1$ and $L_2$ are defined and satisfy $\Lambda_{L_1}=\Lambda_{L_2}$. Assume that $F:\Sigma_1\to \Sigma_2$ is a morphism of solutions in the sense that
\[
 F^*U_1=U_2, 
\]
for all $U_\beta$, $\beta=1,2$, that solve
\[
 L_\beta U_\beta=0, \quad U_\beta|_{\p \Sigma} =f
\]
for some $f\in C^\infty(\p \Sigma)$. Assume also that the formal Taylor series of $g_\beta$ and $q_\beta$ agree in $g_\beta$-boundary normal coordinates. 

Then the coordinate representation of $F$ in boundary normal coordinates is the identity mapping to infinite order on the boundary: $F(x^1,\ldots, x^n)=\text{Id} + O(x_n^\infty)$.
\end{lemma}
We have placed the proof of the lemma in Appendix \ref{appx:proof_of_lemma}. We remark for possible future references that the lemma generalizes to higher dimensions and for more general second order elliptic operators.

Recall that $h^{(1)}=\p_u|_{u=0} \text{Tr}(g_u^{-1}\p_sg_u)$. Thus, by the assumptions of Theorem \ref{thm:main} that we are proving, the formal Taylor series of $h_1^{(1)}$ and $h_2^{(1)}$ agree in associated boundary normal coordinates. The same holds for $g_1$ and $g_2$. By \eqref{eq:morphism_of_sols}, the mapping $F$ is morphism of solutions. Consequently, by the above lemma 
\begin{equation}\label{eq:F_id_to_infinite}
F=\text{Id} + O(x_n^\infty) \text{ in boundary normal coordinates on } \{x_n=0\}.
\end{equation}
Especially we have $F_*\p_{\nu_1}=\p_{\nu_2}$ and $v_1$ and $v_2$ agree to infinite order in associated boundary normal coordinates.


%
%
%

\subsection{Step 2: Second linearization} Next we will first show by using the second linearization that
\[
 k_1^{(1)}=c\tilde k_2^{(1)},
\]
where
\[
		 \tilde k_2^{(1)}=F^*k_2^{(1)}.
		\]
Here $F$ is the conformal mapping of the last section. We note that $k_\beta^{(1)}$, $\beta =1,2$, are symmetric $2$-tensor fields on $\Sigma_\beta$. Then we will also show that $h_1^{(2)}=F^*h_2^{(2)}$ and $w_1^{(jk)}=F^*w_2^{(2)}$.

By Lemma \ref{Lem:Integral identity_2nd}, the associated integral identity   of the second linearization is
\begin{multline}\label{second_integral_id_proof}
	 \int_{\p \Sigma} f_m \s \p^2_{\eps_j \eps_k}\big|_{\epsilon=0} \Lambda_{g_\beta} (f_\epsilon) \, dS_\beta =\int_{\Sigma_\beta} v^m k_\beta^{(1)}(\nabla v_\beta^k,\nabla v_\beta^j\big)dV_\beta+\int_{\Sigma_\beta}v_\beta^k k_\beta^{(1)}(\nabla v_\beta^j,\nabla v_\beta^m\big)dV_\beta\\
 +\int_{\Sigma_\beta} v_\beta^j k_\beta^{(1)}(\nabla v_\beta^k,\nabla v_\beta^m\big)dV_\beta -\frac{1}{2}\int_{\Sigma_\beta} h_\beta^{(2)}v_\beta^jv_\beta^kv_\beta^mdV_\beta  \\
 - \int_{\p \Sigma} v_\beta^m k_\beta^{(1)}(\nu_\beta,\nabla v_\beta^{(j})v_\beta^{k)}dS_\beta,
 \end{multline}
 where $\beta=1,2$ refers to quantities on $\Sigma_\beta$ and $j,k,m\in \{1,\ldots,4\}$.    
 We change variables in the terms on the right-hand side of \eqref{second_integral_id_proof}  with $\beta = 2$ by using the conformal mapping $F$. 
		Using also $ v^{j}=v_2^{j}\circ F$,
		we have
		\begin{multline*}
		 \int_{\Sigma_2} v_2^m k_2^{(1)}(\nabla v_2^k,\nabla v_2^j\big)dV_2=\int_{\Sigma_1} F^*\left(v_2^m k_2^{(1)}(\nabla v_2^k,\nabla v_2^j)\right)F^*dV_2 \\
		 =\int_{\Sigma_1} c v^m \tilde k^{(1)}(\nabla v^k,\nabla v^j)dV.
		\end{multline*}
		Here the factor $c$ is due to the volume form. We also used the fact that $k_2^{(1)}$ is a tensor on $\Sigma_2$ so that $F^*$ acts on it by the usual coordinate transformation rules.  We have similarly for other terms in the \eqref{second_integral_id_proof}. Consequently, using the assumptions $k_1^{(1)}|_{\p \Sigma}=k_2^{(1)}|_{\p \Sigma}$ and $\Lambda_{g_1} (f_\epsilon) =\Lambda_{g_2} (f_\epsilon)$, the fact that $c|_{\p \Sigma}=1$,  and subtracting the identities \eqref{second_integral_id_proof} on $\Sigma=\Sigma_1$ and $\Sigma_2$ give 
\begin{multline}\label{eq:2nd_ord_recovery_proof}
 0=\int_{\Sigma} v^1 (k^{(1)}_1-c\tilde k^{(1)}_2)(\nabla v^2,\nabla v^3\big)dV+\int_{\Sigma}v^k (k^{(1)}_1-c\tilde k^{(1)}_2)(\nabla v^1,\nabla v^3\big)dV_1\\
 +\int_{\Sigma} v^3 (k^{(1)}_1-c\tilde k^{(1)}_2)(\nabla v^1,\nabla v^2\big)dV -\frac{1}{2}\int_{\Sigma} (h^{(2)}_1-c\tilde h^{(2)}_2)v^1v^2v^3dV.
\end{multline}
Here we also chose $(j,k,l)=(1,2,3)$. We note that since $F^*g_2=c g_1$ and $g_1=g_2$ and $F$ is the identity to infinite order on $\p \Sigma$ in associated boundary normal coordinates by \eqref{eq:F_id_to_infinite}, we have that 
\begin{align*}\label{eq:c_is_1_to_infinite}
 c&=1 \text{ to infinite order on } \p\Sigma \\
 k^{(1)}_1&=c\tilde k^{(1)}_2 \text{ to infinite order on } \p\Sigma.
\end{align*}

Let $z_0\in \Sigma$ be fixed. Let $a,b\in \R$ and let us choose the solutions as
 \begin{equation}\label{eq:sols_for_second_proof}
 \begin{split}
  v^1&=e^{\Theta_1/h}(a+r_h) \\
  v^2&=e^{\Theta_2/h}(a+r_h) \\
  v^3&=e^{\Theta_3/h}(\overline a+\tilde r_h),
\end{split}
\end{equation}
as in Section \ref{sec:CGOs}. 
Here in local holomorphic coordinates near $z_0=0$ 
\begin{align*}
\begin{split}
\Theta_1&=\frac{1}{2}z^2+O(z^3) \text{ is holomorphic }, \quad \Theta_2=\Theta_1, \\
\Theta_3&=-\z^2+O(\z^3) \text{ is antiholomorphic}.
\end{split}
\end{align*}
The amplitude $a$ is $1$ to high order at $z_0$ and $0$ to high order at other possible critical points of the phase functions $\Theta_j$. 
With these solutions, we note that on the right hand side of \eqref{eq:2nd_ord_recovery_proof} the first three terms are of size $O(1)$ by stationary phase \cite[Theorem 7.7.5]{hormander2015analysis} and by Proposition \ref{prop:asymptotics_for_the_2nd_lin}. The last term in \eqref{eq:2nd_ord_recovery_proof} is of size $O(h)$ by stationary phase and the $L^p$ estimates for $r_h$ and $\tilde r_h$, see \eqref{eq:rh_L2_estim}--\eqref{eq:rh_Lr_estim}. Therefore we may first focus on the three first terms. By Proposition \ref{prop:asymptotics_for_the_2nd_lin} it suffices to compute the asymptotics of 
\[
 \int_{\Sigma} \hat v^1 K(\nabla \hat  v^2,\nabla \hat v^3\big)dV+\int_{\Sigma}\hat v^2 K(\nabla \hat v^1,\nabla \hat v^3\big)dV+\int_{\Sigma}\hat v^3 K(\nabla\hat v^1,\nabla \hat v^2\big)dV,
\]
where  
\[
 K:=k^{(1)}_1-c\tilde k^{(1)}_2
\]
and where 
\[
\hat v^1=\hat v^2=e^{\Theta_2/h}a \ \text{ and } \ \hat v^3=e^{\Theta_3/h}\overline a 
\]
 are the leading order parts of the CGOs $v^j$. Let us denote 
\[
 \Theta =\Theta_1+\Theta_2+\Theta_3.
\]
Note that $\Theta$ is purely imaginary and has a nondegenerate critical point at $z_0$.

Since $K$ vanishes to high order on the boundary of $\Sigma$ and $a$ vanishes to high order at other possible critical points, it suffices to calculate the asymptotics of 
\begin{multline}\label{eq:final_2nd_lin_sum}
 0=\int_{U} \varphi e^{\Theta/h} K(\nabla \Theta_2,\nabla \Theta_3\big)dV+\int_{U} \varphi  e^{\Theta/h} K(\nabla \Theta_1,\nabla \Theta_3\big)dV \\
 +\int_{U} \varphi  e^{\Theta/h} K(\nabla \Theta_1,\nabla \Theta_2\big)dV,
\end{multline}
where $U$ is a holomorphic coordinate chart containing no other critical points than $z_0$.   (See \cite[Proof of Proposition 3.1]{carstea2022inverse} for justification of this statement.) 
The terms in \eqref{eq:final_2nd_lin_sum} correspond to the cases when all gradients hit the factors $e^{\Theta_j/h}$, yielding the large coefficient $h^{-1}$,  and none the amplitudes $a$ or $\overline a$. Here also $\varphi$ is compactly supported in $U$ and equals $1$ to high order at the critical point $z_0$.


For the following computations it is beneficial to write
\[
 \text{Hess}(z^2):=H=S+iA,
\]
where 
\[
 S=2
  \left[ {\begin{array}{cc}
   1 & 0 \\
   0 & -1 \\
  \end{array} } \right] \text{ and } A  =2\left[ {\begin{array}{cc}
   0 & 1 \\
   1 & 0 \\
  \end{array} } \right].
\]
Then 
\[
 \text{Hess}_{z=0}(\Theta)=(S+iA)-(S-iA)=2i A.
\]
We also write 
\begin{align*}
 H^{1}&=H^{2}=\text{Hess}_{z=0}(\Theta_1)=\text{Hess}_{z=0}(\Theta_2)=\frac{1}{2}H, \\
 H^{3}&=\text{Hess}_{z=0}(\Theta_3)=-\overline H.
\end{align*}
%
%
%
%
By H\"ormander \cite[Theorem 7.7.5]{hormander2015analysis}  we have 
\begin{multline*}
 \int_{U} \varphi e^{\Theta/h} K(\nabla \Theta_2,\nabla \Theta_3\big)dV=h^2C_\Theta (A^{-1})_{ab}\p_a\p_b(K_{cd}\p_c\Theta_2 \p_d\Theta_3)|_{z=0}+O(h^3) \\
 =h^2C_\Theta \text{Tr}(H^2A^{-1}H^3K)+O(h^3)
 \end{multline*}
(Looking first at \cite[Lemma 7.7.3]{hormander2015analysis} might be helpful in understanding the reason for the above identity.) 
Here the coefficient $C_\Theta\neq 0$ only depends on $\Theta$ and we use Einstein summation over repeated indices. We have similarly for the other three first terms of \eqref{eq:final_2nd_lin_sum}. Consequently, by taking the limit $h\to 0$ yields 
\[
 \text{Tr}(H^2A^{-1}H^3K)+\text{Tr}(H^1A^{-1}H^3K)+\text{Tr}(H^1A^{-1}H^2K)=0.
\]
Note that $A^{-1}=A/2$. The above is equivalent to
%
\begin{multline}\label{eq:simplfied_identity_final}
\text{Tr}(AHKH)-\text{Tr}(A\overline HK H)=0 \Leftrightarrow  \text{Tr}(A(H-\overline H)KH)=0 \\
\Leftrightarrow  \text{Tr}(A^2KH)=0 \Leftrightarrow  \text{Tr}(KH)=0 
\Leftrightarrow  \text{Tr}(KS)=0 \text{ and } \Leftrightarrow  \text{Tr}(KA)=0.
\end{multline}
Using the last conditions we thus have
\[
 K_{11}=K_{22} \text{ and } K_{12}=-K_{21}.
\]
Since $K$ is also symmetric, we have obtained
\[
 K(z_0)=K_{11}(z_0)I_{2\times 2}.    
\]
%
%

We show that $K_{11}(z_0)$ vanishes. For this, we use the fact that we are dealing with minimal surfaces, which have the property that the mean curvature
\[
 h=\tr(g^{-1}\p_sg)
\]
vanishes, see \eqref{eq_h_equiv_0}. By using that $k_2^{(2)}$ is a tensor, $F^*g_2=cg_1$ and trace of the $(1,1)\s $-tensor $g_2^{-1}k_2^{(1)}$ is invariant, at $z_0$ we have 
\begin{multline*}
 K_{11} \tr(g^{-1})=\tr(g^{-1}K)=\tr(g_1^{-1}K)=\tr(g^{-1}k_1^{(1)})-\tr(g^{-1}c\tilde k_2^{(1)}) \\ 
 =\tr(g^{-1}\p_sg_1)-\tr(c^{-1}(F^*g_2)cF^* k_2^{(1)}) = \tr(g_2^{-1}k_2^{(1)})|_F=0.
\end{multline*}
Thus $K(z_0)=0$. Repeating the argument for all points of $\Sigma$ shows that
\begin{equation}\label{eq_k_determined}
 k_1^{(1)}= c\tilde k_2^{(1)} \text{ on } \Sigma_1
\end{equation}
as claimed. 

Next we show that $h_1^{(2)}=c\tilde h_2^{(2)}$  on  $\Sigma_1$. 
By \eqref{eq_k_determined}, the identity \eqref{eq:2nd_ord_recovery_proof} now reads
\[
  0=\int_{\Sigma} (h^{(2)}_1-c\tilde h^{(2)}_2)v^1v^2v^3dV.
\]
Let $z_0\in \Sigma$. We choose $v_1$, $v_2$ and $v_3$ as before in \eqref{eq:sols_for_second_proof}. Thus
\[
 0=\int_{U} \varphi e^{\Theta/h} (h^{(2)}_1-c\tilde h^{(2)}_2)dV,
\]
where $\varphi$ is $1$ to high order at $z_0$. Arguing similarly as before, we obtain by stationary phase \cite[Theorem 7.7.5]{hormander2015analysis} that
 $h^{(2)}_1(z_0)=c(z_0)\tilde h^{(2)}_2(z_0)$. 
Repeating the argument for all points of $\Sigma$ yields
\[
 h^{(2)}_1= c\tilde h^{(2)}_2 \text{ on } \Sigma_1.
\]

Recall that the second linearizations $w_\beta^{jk}$, $\beta =1,2$,  
 satisfy 
 \begin{equation*}
 \begin{aligned}
		\begin{cases}
  (\Delta_{g_\beta}+h_\beta^{(1)}/2)w_\beta^{jk}+P_\beta^{(j}v_\beta^{k)}+k_\beta^{(1)}(\nabla v_\beta^{j},\nabla v_\beta^{k})+\frac{1}{2}h_\beta^{(2)}v_\beta^{j}v_\beta^{k}= 0& \text{ in } \Sigma \\
  w_\beta^{jk}=0
			&\text{ on }\p \Sigma,
  		\end{cases}
	\end{aligned}
 \end{equation*}
 where $P_\beta^j=-\nabla^{g_\beta}\cdot_\beta (v_\beta^jg_\beta k_\beta^{(1)}\nabla)$. Since we have recover all the coefficients of the above equation up to a conformal scaling factor, we see by uniqueness of solutions to elliptic equations that $w_1^{jk}=F^*w_2^{jk}=:\tilde w_2^{jk}$. All in all, we have shown by using the second linearization that
 \begin{align*}
  k_1^{(1)}=c\tilde k_2^{(1)} , \quad   h^{(2)}_1&=c\tilde h^{(2)}_2, \quad w_1^{jk}=\tilde w_2^{jk},
 \end{align*}
 for all $j,k=1,\ldots,4$.

\subsection{Step 3: Third linearization}
To complete the proof, we are left to show that the conformal factor $c=1$ on $\Sigma_1$. For this, we use the third linearization. By Lemma \ref{Lem:Integral identity_3rd}, the corresponding integral identity is 
	\begin{multline}\label{eq:third_integral_id_proof}
	 \int_{\p \Sigma} f_m \s \p^3_{\eps_j \eps_k\eps_l}\big|_{\epsilon=0} \Lambda (f_\epsilon) \, dS_g = \int_{\Sigma}g(\nabla v^{j},\nabla v^k)g(\nabla v^{l},\nabla v^m) dV  \\
	 +\int_{\Sigma}g(\nabla v^{j},\nabla v^l)g(\nabla v^{k},\nabla v^m) dV+\int_{\Sigma}g(\nabla v^{l},\nabla v^k)g(\nabla v^{j},\nabla v^m) dV \\
	 +H+R+B,
	\end{multline}
	where
	\begin{equation*}
 \begin{split}
 H&=-\int_{\Sigma}v^{(j}v^{k} k^{(2)}(\nabla v^{l)},\nabla v^m) dV +\int_{\Sigma}v^m g(\nabla (d^{-1}d^{(2)}v^{(j}v^k),\nabla v^{l)})dV \\
 &\qquad-\int_\Sigma v^m k^{(2)}(\nabla v^{(j},\nabla v^{k})v^{l)}dV-\frac{1}{2}\int_\Sigma v^mv^{j}v^{k}v^{l}h^{(3)}dV,\\
 R&=-\int_{\Sigma}w^{(jk}k^{(1)}(\nabla v^{l)},\nabla v^m)dV-\int_\Sigma k^{(1)}(\nabla v^m,\nabla v^{(j})w^{kl)}dV \\
 &\qquad-\int_{\Sigma}v^mk^{(1)}(\nabla v^{(j},\nabla w^{kl)})dV-\frac{1}{2}\int_\Sigma v^mg(\nabla v^{(j},\nabla v^{k})v^{l)}h^{(1)}dV  \\
 &\qquad \qquad -\frac{1}{2}\int_\Sigma v^mw^{(jk}v^{l)}h^{(2)}dV,   \\
 B&=\int_{\p \Sigma} v^mv^{(j}v^k g(\nu,k^{(2)}\nabla v^{l)})dS+\int_{\p \Sigma} v^mw^{(jk} g(\nu,k^{(1)}\nabla v^{l)})dS\\
&\qquad-\int_{\p \Sigma} v^mg(\nabla v^{(j},\nabla v^k)\p_\nu v^{l)}dS+\int_{\p \Sigma} v^mv^{(j} k^{(1)}(\nu,\nabla w^{kl)})dS.
\end{split}
\end{equation*}
The identity holds for any $j,k,m\in \{1,\ldots,4\}$ and for quantities on both $\Sigma=\Sigma_1$ and $\Sigma=\Sigma_2$. 

As already remarked after Lemma \ref{Lem:Integral identity_3rd} we will be able to disregard all the terms $H$, $R$ and $B$. Let us explain how. First of all, the terms $B$ are boundary terms, which we assume to be known. We have also shown that $v_1^{j}=\tilde v_2^{j}$ and $w_1^{jk}=\tilde w_2^{jk}$ on $\Sigma_1$. Especially, $v_1^{j}= v_2^{j}$ and $w_1^{jk}=w_2^{jk}$ to infinite order on the boundary by \eqref{eq:F_id_to_infinite}. Thus, when we subtract the integral identity \eqref{eq:third_integral_id_proof} on $\Sigma_1$ from that on $\Sigma_2$, the boundary terms $B$ will cancel. 

Next, we note that we have recovered all the quantities in $R$, up to a possible conformal factor. Let us see how the conformal factor behaves when we make change of variables in integration. Let us consider the simplest term
\[
 \int_{\Sigma_2} v_2^mw_2^{(jk}v_2^{l)}h_2^{(2)}dV_2
\]
in $R$ on $\Sigma_2$. This term equals
\[
 \int_{\Sigma_1} F^*(v_2^mw_2^{(jk}v_2^{l)}h_2^{(2)})F^*dV_2=\int_{\Sigma_1} v^mw^{(jk}v^{l)}c^{-1}h_1^{(2)}cdV_1=\int_{\Sigma_1} v^mw^{(jk}v^{l)}h_1^{(2)}dV_1.
\]
Thus the conformal factor $c$ cancels out when we make a change of variables. The other terms of $R$ behave the same way. Consequently, when we subtract the terms of $R$ on $\Sigma_1$ from those on $\Sigma_2$, they will cancel out.

Finally we note that while we have not recovered the terms of $H$, those terms contain at most $2$ derivatives on $v^j$ and do not  contain any $w^{jk}$. As before, we are going to use the CGOs of Section \ref{Section 2} as our solutions $v^j$. Consequently, by stationary phase these terms will be of size $O(1)$, while the first three terms on the right hand side of \eqref{eq:third_integral_id_proof} will be of size $O(h^{-1})$. We will provide more details to this statement below, after we have made specific choices for the CGOs $v^j$. All in all, we will be able to disregard all the terms $H$, $R$ and $B$.

Let us then consider the first three terms on the right hand side of \eqref{eq:third_integral_id_proof}. These will be the leading order terms when we fix the solutions $v^j$ as CGOs. By making a change of variable in integration, we have 
\begin{multline*}
 \int_{\Sigma_2}g_2(\nabla v_2^{j},\nabla v_2^k)g_2(\nabla v_2^{l},\nabla v_2^m) dV_2=\int_{\Sigma_1}F^*\left(g_2(\nabla v_2^{j},\nabla v_2^k)g_2(\nabla v_2^{l},\nabla v_2^m)\right) F^*dV_2   \\
 =\int_{\Sigma_1}F^*g_2(\nabla v^{j},\nabla v^k)F^*g_2(\nabla v^{l},\nabla v^m)cdV_1=\int_{\Sigma_1}c^{-1}g_1(\nabla v^{j},\nabla v^k)g_1(\nabla v^{l},\nabla v^m)dV_1.
\end{multline*}
Here we used that $F^*g_2=cg_1$ (or $F^*g_2=c^{-1}g_1$ when $F^*g_2$ is acting on $1$-forms as above) and that $v^j=F^*v_2^j$. Thus, by what we have argued, subtracting the identity \eqref{eq:third_integral_id_proof} on the $\Sigma_1$ from that on $\Sigma_2$ and using $\Lambda_{g_1} (f_\epsilon) =\Lambda_{g_2} (f_\epsilon)$ we have
\begin{multline}\label{eq:withH1H2}
H_1-H_2=\int_{\Sigma_1}(1-c^{-1}) \Big[g_1(\nabla v^1\cdot \nabla v^2)g_1(\nabla v^3\cdot \nabla  v^4)+g_1(\nabla v^1\cdot \nabla v^3)g_1(\nabla v^2\cdot \nabla  v^4) \\
 +g_1(\nabla v^1\cdot \nabla v^4)g_1(\nabla v^2\cdot \nabla  v^3)\Big]dV_1.
 \end{multline}
 Here we also chose $(j,k,l,m)=(1,2,3,4)$. 

Let $z_0\in \Sigma$. We choose the four solutions as
\begin{align*}
\begin{split}
 v_1&=v_3=e^{\Phi/h}(a+r_h) \\
 v_2&=v_4=e^{-\overline \Phi/h}(\overline a+\tilde r_h) ,
 \end{split}
\end{align*}
where $\Phi$ is a holomorphic Morse function with a critical point at $z_0$. These solutions are as in \eqref{eq:sols_for_4th}. 
The holomorphic function $a$ such that it has the expansion $a(z)=1+O(z^N)$ near the critical point $z_0$ and that $a$ vanishes to high order at all other possible critical points of $\Phi$. As $g_1=g_2$ to infinite order on $\p \Sigma$ and $F^*g_2=g_1$, we have that $c^{-1}$ is $1$ to infinite order on the boundary. Consequently, by Proposition \ref{prop:asymptotics_for_the_3rd_lin}, we have
\begin{multline}\label{eq:expansion_for_3rd_lin_proof}
\int_{\Sigma}(1-c^{-1})\Big[g(\nabla v_1\cdot \nabla v_2)g(\nabla v_3\cdot \nabla  v_4)+g(\nabla v_1\cdot \nabla v_3)g(\nabla v_2\cdot \nabla  v_4) \\
 +g(\nabla v_2\cdot \nabla v_3)g(\nabla v_1\cdot \nabla  v_4)\Big]dV=h^{-1}(1-c^{-1}(z_0))+o(h^{-1}).
 \end{multline}

Let us complete the proof by arguing that the terms $H_1$ and $H_2$ are negligible. Indeed, as they contain only $v^j$ and up to $2$ derivatives of them, and  everything else in these terms is independent of $h$, these terms are of size $O(1)$ by \cite[Theorem 7.7.5]{hormander2015analysis} and the $L^p$ estimates \eqref{eq:rh_L2_estim}--\eqref{eq:rh_Lr_estim} for $r_h$ and $\tilde r_h$ and their gradients. Combining this fact with \eqref{eq:withH1H2} and \eqref{eq:expansion_for_3rd_lin_proof} shows that
\[
 h^{-1}(1-c^{-1}(z_0))+o(h^{-1})=0.
\]
Thus $c(z_0)=1$. Repeating this argument for all $z_0\in \Sigma_1$ shows that
\[
 F^*g_2=g_1 \text{ on } \Sigma_1.
\]
Finally noting that the scalar second fundamental form reads $\eta_{ab}(x)=\p_s|_{s=0}g_{ab}(x,s)$ in Fermi-coordinates (see e.g. \cite{lassas2016calder} for the formula), knowing $k^{(1)}$ and $g$ determines $\eta$. Thus also $F^*\eta_2=\eta_1$. This concludes the proof of Theorem \ref{thm:main}.

\subsection{Proof of Theorem \ref{thm:main2}}
We note that in Steps 2 and 3 of the proof of Theorem \ref{thm:main}, we did not use that $\Sigma_1$ and $\Sigma_2$ are a domain in $\R^2$ as sets. Those steps hold for general Riemannian surfaces.  Using this observation, we prove our other main theorem.
\begin{proof}[Proof of Theorem \ref{thm:main2}]
 By assumption $(\Sigma_1,g_1)$ and $(\Sigma_2,g_2)$ are general Riemannian conformal Riemannian surfaces: $g_2=cg_1$. Since $g_2=g_1$ to infinite order on the boundary by assumption, we have that $c=1$ to infinite order on the boundary. 
 
 By Lemma \ref{lem:high_ord_lin}, the first linearization of the minimal surface equation is  
\begin{equation}\label{eq:first_lin_proof_2}
	\begin{aligned}
		\begin{cases}
			(\Delta_{g_\beta}+h_{\beta}^{(1)}/2)v_\beta^{j}=0 
			& \text{ in } \Sigma_\beta,
			\\
			v_\beta^{j}=f_j
			&\text{ on }\p \Sigma,
		\end{cases}
	\end{aligned}
\end{equation}
$\beta=1,2$ and $j=1,\ldots 4$.
Using the $\Delta_{g_2}$ is conformally invariant, we have that the DN maps of the equations
\begin{align*}
 (\Delta_{g_1}+ch_{2}^{(1)}/2)v_2=0 \text{ and }  (\Delta_{g_1}+h_{1}^{(1)}/2)v_1=0
\end{align*}
agree on $\Sigma_1$. By \cite{guillarmou2011calderon} (or \cite{guillarmou2011identification}) we have 
\[
h_{1}^{(1)}=ch_{2}^{(1)} \text{ on } \Sigma_1.
\]
Consequently, by uniqueness of solutions to the Dirichlet problem \eqref{eq:first_lin_proof_2}, we also have
\[
 v_1^{j}=v_2^{j} \text{ on } \Sigma_1.
\]
By setting $F\equiv \text{Id}$, we have obtained the conclusions of Step 1 of the proof of Theorem \ref{thm:main}. Since Steps 2 and 3 hold without the assumption that $\Sigma_1$ and $\Sigma_2$ are a domain in $\R^2$ as sets, the rest of the proof is identical to that of Theorem \ref{thm:main}.  Thus $c=1$ on $\Sigma_1$.

\end{proof}

\begin{remark}\label{rem:generalization_of_earlier}
 Assume that $(\Sigma,g)$ is a general $2$-dimensional minimal surface satisfying
 \begin{equation}\label{eq:q_condition}
 \frac{d}{d s}\Big|_{s=0}\text{Tr}(g_s^{-1}\p_sg_s)=0
 \end{equation}
 in corresponding Fermi coordinates. Then, the first linearization \eqref{linear_eq} is just $\Delta_gv=0$. Consequently, we may recover $(\Sigma,g)$ up to a conformal mapping in Step $1$ of the proof of Theorem \ref{thm:main}. This follows for example from \cite[Theorem 5.1]{lassas2018poisson}. Repeating then Steps $2$ and $3$  recover $(\Sigma, g)$ (up to an isometry). This means that if \eqref{eq:q_condition} holds,  we may remove the topological assumption about $\Sigma$ in Theorem \ref{thm:main}. Especially, if $N=\Sigma\times\R$ and $\bar g = ds^2+g$, in which case obviously \eqref{eq:q_condition} is true, we may recover $(\Sigma,g)$. Thus, Theorem \ref{thm:main} (or its proof) generalizes the main result of \cite{carstea2022inverse}.
 \end{remark}

\appendix
\section{Proof of lemmas}
\begin{lemma}\label{lem:log_lemma}
 Assume that $\deg(f)=0$ and that $\psi$ is Morse. For all $\eps>0$ small enough
 \[
  \op_\psi^{-1}f=O_{L^2}(h^{\frac{1}{2}+\epsilon}) \text{ and } \p_\psi^{-1}f=O_{L^2}(h^{\frac{1}{2}+\epsilon}).
 \]
\end{lemma}
\begin{proof}
By changing variables, we may assume that $f$ is supported in a ball of radius $1$. Let $\chi$ be a cutoff function, which is identically $0$ on a ball of radius $1$ and $1$ outside the ball of radius $2$. Then
\[
 \chi_h:=\chi(z/h^{1/2})
\]
is a cutoff function, which is $0$ on a ball of radius $h^{1/2}$ and $1$ outside a ball of radius $2h^{1/2}$.  We have
\begin{align}
 \op_\psi^{-1}f=\op^{-1} (e^{-2i\psi/h} f)=\op^{-1} (e^{-2i\psi/h} \chi_h f)+\op^{-1} (e^{-2i\psi/h} (1-\chi_h)f).
\end{align}
 Let $1<q<2$.  We have
\begin{multline*}
 \norm{\op^{-1} (e^{-2i\psi/h} (1-\chi_h)f)}_{L^2}\lesssim  \norm{\op^{-1} (e^{-2i\psi/h} (1-\chi_h)f)}_{L^{q^*}}\lesssim  \norm{\op^{-1} (e^{-2i\psi/h} (1-\chi_h)f)}_{W^{1,q}}\\[5pt]
\lesssim \norm{(1-\chi_h)f}_{L^q}\lesssim \norm{f}_{L^\infty}\text{Vol}(B(0,h^{1/2}))^{1/q}\lesssim h^{1/q} \norm{f}_{L^\infty}.
\end{multline*}
In the above, the first inequality follows by H\"older's inequality (since the Sobolev conjugate exponent $q^*=2q/(2-q)>2$), the  second is Sobolev's inequality, and the third is the boundedness of $\op^{-1}$ from $L^q$ to $W^{1,q}$ (see, e.g. \cite{guillarmou2011identification}).

Then we integrate by parts to have
\begin{multline*}
 \op^{-1} (e^{-2i\psi/h} \chi_h f)=\int e^{-2i\psi/h}\chi_h f(z-w)^{-1}dz=-\frac{i}{2}h\int e^{-2i\psi/h}\op\left(\chi_h\frac{f}{\op \psi}(z-w)^{-1}\right) \\
 =-\frac{i}{2}h e^{i\psi/h} \chi_h\frac{f}{\op \psi}-\frac{i}{2}h\int e^{i\psi/h}\op \left(\chi_h\frac{f}{\op \psi}\right)(z-w)^{-1}.
\end{multline*}
Using the boundedness of $\op^{-1}$ again we have
\begin{multline}\label{eq86}
 \norm{\op^{-1} (e^{-2i\psi/h} \chi_h f )}_{L^2}\lesssim h\norm{\chi_h\frac{f}{\op \psi}}_{L^2}+h\norm{\op \left(\chi_h\frac{f}{\op \psi}\right)}_{L^q}\\
 \lesssim h\norm{\chi_h\frac{f}{\op \psi}}_{L^2}+h\norm{\left(\op \chi_h\right)\frac{f}{\op \psi}}_{L^q} +h\norm{\chi_h\frac{\op f}{\op \psi}}_{L^q}+h\norm{\chi_h\frac{f}{(\op \psi)^2}}_{L^q}.
\end{multline}
The first term on the right of \eqref{eq86} is bounded as
\[
 \norm{\chi_h\frac{f}{\op \psi}}_{L^2}^2\leq \norm{f}_{L^\infty}^2\norm{\frac{\chi_h}{\op\psi}}_{L^2}^2\lesssim \norm{f}_{L^\infty}^2\int_{1\geq|z|\geq h^{1/2}}|z|^{-2}\lesssim \int_{h^{1/2}}^1r^{-1} dr=\frac{1}{2}|\log h|.
\]
Here we used that $\psi$ is Morse to have $\abs{(\op \psi)^{-1}}\lesssim \abs{z}^{-1}$.

The second term on the right of \eqref{eq86} is bounded as
\begin{multline*}
 \norm{(\op\chi_h)\frac{f}{\op \psi}}_{L^q}^q=\int_{B(0,1)} \abs{\op\chi_h}^q\abs{\frac{f}{\op \psi}}^q\lesssim \norm{(\op \chi_h)}_{L^\infty}^q\int_{h^{1/2}}^{2h^{1/2}} r\left|\frac{f}{r}\right|^q dr \\
 \lesssim h^{-q/2} \int_{h^{1/2}}^{2h^{1/2}} r^{1-q}dr \lesssim h^{-q/2}h^{\frac{1}{2}(2-q)}=h^{1-q}.
\end{multline*}
Here we additionally used  
\[
 \norm{\op \chi_h}_{L^\infty}\leq h^{-1/2}\norm{\op \chi}_{L^\infty}\lesssim h^{-1/2}.
\]

%

Since $\deg(f)=0$, we have $f(z)=z^k\overline z^l+O(\abs{z})$ with $k+l=0$. Thus
\[
 \op f= C\s z^k\overline z^{l-1} + O(1) \text{ and } \abs{\op f}\lesssim \abs{z}^{-1}.
\]
Using this, the third term on the right of \eqref{eq86} is bounded as 
\begin{equation}\label{eq:log_estimate}
 \norm{\chi_h\frac{\op f}{\op \psi}}_{L^q}^q
\lesssim \int_{\{1\geq\abs{z}\geq h^{1/2})\}} \abs{z}^{-2q}
 =  \int_{h^{1/2}}^1  r^{1-2q}dr\lesssim h^{1-q}.
\end{equation}
The fourth term on the right of \eqref{eq86} is bounded as
\begin{equation*}
 \norm{\chi_h\frac{f}{(\op \psi)^2}}_{L^q}^q\lesssim \norm{f}_{L^\infty}^q\int_{\{1\geq\abs{z}\geq h^{1/2})\}} \abs{z}^{-2q} \lesssim h^{1-q}.
\end{equation*}

Combining everything we have
\[
 \norm{\op_\psi^{-1}(f)}_{L^2}\lesssim h^{1/q}+h(|\log h|^{1/2}+h^{\frac{1}{q}-1}).
\]
Since $1<q<2$, the result follows.
\end{proof}

\section{Calculations}\label{appx:calculations}
We record computations used in Section \ref{Section 2}. For the below calculations, we note that
\[
 \frac{d^2}{d\eps_j\eps_k}\Big|_{\eps=0}(1+\abs{\nabla u}^2_u)^{-1/2}=-g(\nabla v_j,\nabla v_k).
\]

The calculation behind the equation \eqref{eq:third_deriv_of_f} is the following
\begin{multline*}
 \frac{d^3}{d\eps_{j}d\eps_kd\eps_l}\Big|_{\eps=0}f(u,\nabla u) \\
 =\frac{1}{2}\frac{1}{(1+\abs{\nabla u}^2_{u})^{1/2}}|_{\eps=0}\left(\p_{\eps_{jkl}}|_{\eps=0}k_u^{(1)}(\nabla u,\nabla u)\right)+\frac{1}{2}\p_{\eps_{jkl}}|_{\eps=0}\left((1+\abs{\nabla u}^2_{u})^{1/2}h_u\right) \\
 =\frac{1}{2}\p_{\eps_{jk}}|_{\eps=0}\left[\left(k_u^{(2)}(\nabla u,\nabla u)u_l+2k_u^{(1)}(\nabla u_l,\nabla u)\right)\right] \\
 +\frac{1}{2}\p_{\eps_{jk}}|_{\eps=0}\left[(1+\abs{\nabla u}^2_{u})^{-1/2}g_u(\nabla u_l,\nabla u)h_u+(1+\abs{\nabla u}^2_{u})^{1/2}h_u^{(1)}u_l \right] \\
 =\frac{1}{2}\p_{\eps_{j}}|_{\eps=0}\Big[2k_u^{(2)}(\nabla u_k,\nabla u)u_l+2k_u^{(2)}(\nabla u_l,\nabla u)u_k+2k_u^{(1)}(\nabla u_{kl},\nabla u)+2k_u^{(1)}(\nabla u_{l},\nabla u_k)\Big] \\
 +\frac{1}{2}\p_{\eps_{j}}|_{\eps=0}\Big[g_u(\nabla u_l,\nabla u_k)h_u+g_u(\nabla u_l,\nabla u)h_u^{(1)}u_k \\ 
 \qquad\qquad\qquad\qquad\qquad\qquad\qquad\qquad\qquad\qquad +g_u(\nabla u_k,\nabla u)h_u^{(1)}u_l+h_u^{(2)}u_lu_k+h_u^{(1)}u_{kl} \Big] \\
 =k^{(2)}(\nabla v^{(j},\nabla v^k)v^{l)}+k^{(1)}(\nabla w^{(jk},\nabla v^{l)})+\frac{1}{2}g(\nabla v^{(j},\nabla v^k)v^{l)}h^{(1)} \\
 +\frac{1}{2}w^{(jk}v^{l)}h^{(2)}+\frac{1}{2}v^{j}v^{k}v^{l}h^{(3)}+\frac{1}{2}w^{jkl}h^{(1)}.
\end{multline*}
In the above calculation $u|_{\eps=0}\equiv 0$ and $h=0$ was used several times.

The calculation behind the formula for \eqref{eq:formula_for_Pjk} is the following
\begin{multline*}
 P^{jk}
 =-\frac{d^2}{d\eps_jd\eps_k}\Big|_{\eps=0}d_u^{-1}\nabla\cdot k_ud_u(1+\abs{\nabla u}^2_u)^{-1/2}\nabla  \\
 =\frac{d}{d\eps_l}\Big|_{\eps=0}\Big[d^{-2}_u d^{(1)}_u u_k\nabla\cdot k_ud_u (1+\abs{\nabla u}^2_u)^{-1/2} \nabla-d^{-1}_u\nabla\cdot k_u^{(1)}u_kd_u (1+\abs{\nabla u}^2_u)^{-1/2} \nabla\\
 -d^{-1}_u\nabla\cdot k_u d_u^{(1)}u_k (1+\abs{\nabla u}^2_u)^{-1/2} \nabla -d^{-1}_u\nabla\cdot k_u d_u \frac{d}{d\eps_k}(1+\abs{\nabla u}^2_u)^{-1/2} \nabla\Big]  \\
 =d^{-2} d^{(2)} v^k v^l\nabla\cdot kd \nabla-d^{-1}\nabla\cdot k^{(2)}v^kv^ld \nabla-d^{-1}\nabla\cdot k^{(1)}w^{kl}d \nabla-d^{-1}\nabla\cdot k d^{(2)}v^kv^l
 \nabla  \\
 +d^{-1}\nabla\cdot kd g(\nabla v^k,\nabla v^l)\nabla \\
 =-k\nabla (d^{-1}d^{(2)}v^k v^l)\cdot \nabla-d^{-1}\nabla\cdot (k^{(2)}v^kv^ld \nabla) \\
 -d^{-1}\nabla\cdot (k^{(1)}w^{kl}d \nabla)
 +d^{-1}\nabla\cdot (kd g(\nabla v^k,\nabla v^l)\nabla).
\end{multline*}

\section{Boundary determination for the isometry}\label{appx:proof_of_lemma} 
\begin{proof}[Proof of Lemma \ref{lem:F_id_infty}]
The proof is similar to that of \cite[Lemma 3.4]{lassas2016calder}. 
 Let $p\in \p M$, let us fix coordinates on a neighborhood $\Gamma$ of $p$ in $\p \Sigma$ and let us consider the associated boundary normal coordinates on $\Sigma_1$ and $\Sigma_2$.  
Let $U_\beta$, $\beta=1,2$, be solutions to $L_\beta U_\beta=0$ with $U_\beta|_{\p \Sigma}=f$.
 Below we denote by $F$, $U_1$ and $U_2$ their coordinate representation in the boundary normal coordinates. 
 
 Since $F$ the preserves boundary data of solutions $U_1$ and $U_2$ by assumption, we have
 \[
  F|_{\p M}=\text{Id}_{\p M}.
 \]
 Since $\Lambda_{L_1}=\Lambda_{L_2}$, we have
 \[
  \p_{x_n}U_1=\p_{x_n}U_2 \text{ on } \{x_n=0\}.
 \]
 It follows that $\nabla U_1=\nabla U_2$ on $\{x_n=0\}$. Here $\nabla$ is the gradient in $\R^n$. Combining this with the assumptions that $F$ is a morphism of solutions, we then have
 \[
  \nabla U_1=\nabla U_2=\nabla(U_1\circ F)=DF^T\nabla U_1|_F=DF^T\nabla U_1 \text{ on } \{x_n=0\}
 \]
 for all solutions $U_1$ to $L_1U_1=0$ in $\Sigma_1$.
 By Runge approximation (see e.g. \cite[Proposition B.1]{lassas2016calder}) and by using linearity of $L_1$, it is possible to choose $\nabla U_1(p)$ freely. Thus 
 \[
  DF|_{x_n=0}(p)=I_{n\times n}.
 \]
 Repeating the argument for all $p\in \Gamma$ shows that $DF|_{x_n=0}(p)=I_{n\times n}$ on $\{x_n=0\}$.
 
 Following the proof of \cite[Lemma 3.4]{lassas2016calder}, we may write
 \[
  \Delta_{g_\beta}+q_\beta=-\p_{x_n}^2+P_\beta,
 \]
 where $P_\beta$ is a partial differential operator containing only first order derivatives in $x_n$ and $x'$ derivatives up to second order and whose coefficients are polynomial expressions of $g_\beta$ and $q_\beta$. Since
 \[
  0=(\Delta_{g_\beta}+q_\beta)U_\beta=-\p_{x_n}^2U_\beta+P_\beta U_\beta,
 \]
 we have by the condition $\nabla U_1=\nabla U_2$ on $\{x_n=0\}$ and the assumption that the formal Taylor series of $g_\beta$ and $q_\beta$ agree to infinite order at $\{x_n=0\}$ in boundary normal coordinates that
 \begin{equation}\label{eq:op_decomp}
  \p_{x_n}^2U_1=P_1 U_1=P_2 U_2=\p_{x_n}^2U_2.
 \end{equation}
 Consequently, we have $\nabla^2 U_1=\nabla^2 U_2$ on $\{x_n=0\}$. 
Then, by using also $DF|_{x_n=0}=I_{n\times n}$ we have
\begin{multline*}
 \p_{x_n}^2U_1=\p_{x_n}^2(U\circ F)= \p_{ab}^2U_2|_F\p_{x_n}F^a\p_nF^b+\p_aU_2|_F\p_{x_n}^2F^a \\
 = \p_{ab}^2U_1\p_{x_n}F^a\p_nF^b+\p_aU_2\p_{x_n}^2F^a=\p_{x_n}^2U_1+\p_aU_1\p_{x_n}^2F^a.
\end{multline*}
Thus
\[
 \p_aU_1\p_{x_n}^2F^a=0 \text{ on } \{x_n=0\}.
\]
As we may choose $\nabla U_1(p)$ freely, we have
 \[
  \p_{x_n}^2F^a(p)=0 \text{ on } \{x_n=0\},
 \]
 for all $a=1,\ldots,n$. Repeating the argument for all $p\in \Gamma$ shows that $\p_{x_n}^2F^a=0$ on $\{x_n=0\}$.
 
The proof is completed by differentiating \eqref{eq:op_decomp} in $x_n$ and by using an induction argument similar to that in \cite[Lemma 3.4]{lassas2016calder}.
 
%
%
%
%

\end{proof}

%
%
%

\bibliography{ref}
\bibliographystyle{abbrv}

\end{document}